\documentclass[reqno,11pt]{amsart}
\usepackage{amsmath,amssymb,amsfonts,amsthm, amscd,indentfirst}
\usepackage{amsmath,latexsym,amssymb,amsmath,
	amscd,amsthm,amsxtra,mathrsfs}
 
\usepackage[top=3cm, bottom=3.5cm, right=2.5cm, left=2.5cm]{geometry}

\usepackage[hyphens]{url}
\usepackage{hyperref}
\usepackage{bookmark}
\usepackage{amsmath,thmtools,mathtools}
\allowdisplaybreaks 
\mathtoolsset{showonlyrefs=true}
\usepackage[msc-links]{amsrefs} 

\newcounter{mparcnt}

\usepackage{fancyhdr}
\usepackage{esint}
\usepackage{enumerate}
\usepackage{xcolor}

\usepackage{pictexwd,dcpic}
\usepackage{graphicx}
\usepackage{caption}

\usepackage{etoolbox}

\usepackage{graphicx}
\usepackage{caption}
\usepackage{slashed}

\usepackage{epsfig,here}
\usepackage{subfigure,here}

\newtheorem{theorem}{Theorem}[section]
\newtheorem{lemma}[theorem]{Lemma}
\newtheorem{proposition}[theorem]{Proposition}
\newtheorem{definition}[theorem]{Definition}
\newtheorem{corollary}[theorem]{Corollary}
\newtheorem{remark}[theorem]{Remark}

\newcommand{\abs}[1]{\lvert#1\rvert}

\newcommand{\norm}[1]{\lVert#1\rVert}

\newcommand{\rd}{{\rm d}}
\newcommand{\rdV}{{\rm dV}}

\newcommand{\rVol}{{\rm Vol}}

\newcommand{\rid}{{\rm id}}

\newcommand{\D}{{\slashed{D}}}
\newcommand{\rRe}{{\rm Re}}
\newcommand{\rIm}{{\rm Im}}
\newcommand{\curl}{{\rm curl}\,}
\newcommand{\ip}{\lrcorner\,}
\newcommand{\w}{\wedge}

\def\<{\langle}
\def\>{\rangle}
\def\S{\mathbb{S}}
\def\R{\mathbb{R}}

\newcommand{\ra}{\rightarrow}

\def\SS{{\mathbb S}}

\newcommand{\eq}[1]{\begin{equation}\allowdisplaybreaks\begin{alignedat}{2} #1 \end{alignedat}\end{equation}}

\numberwithin{equation} {section}

\begin{document}

	
\title[Conformal properties of the zero mode equation]{
Conformal invariants for the zero mode equation
}
\date{\today}


\author{Guofang Wang}
\address{Albert-Ludwigs-Universit\"at Freiburg,
Mathematisches Institut,
Ernst-Zermelo-Str. 1,
D-79104 Freiburg, Germany}
\email{guofang.wang@math.uni-freiburg.de}

\author{Mingwei Zhang}
\address{ Wuhan University, School of Mathematics and Statistics, 430072 Wuhan, China and 
Albert-Ludwigs-Universit\"at Freiburg,
Mathematisches Institut,
Ernst-Zermelo-Str. 1,
D-79104 Freiburg, Germany}
\email{zhangmwmath@whu.edu.cn}

\begin{abstract}

For non-trivial solutions to the zero mode equation on a closed spin manifold
\[\D \varphi=iA\cdot \varphi,\] we first provide a simple proof for the sharp inequality 
\eq{
\norm{A}_{L^n}^2 \ge \frac {n}{4(n-1)} Y(M,[g]),
} where $Y(M,[g])$ is the Yamabe constant of
$(M,g)$, which was obtained by Frank-Loss \cite{Frank_Loss_2024} and Reu\ss {}  \cite{Reuss25}.
Then we classify completely the equality case by proving  that equality holds if and only if $\varphi$ is a Killing spinor, and if and only if $(M,g)$ is a Sasaki-Einstein manifold with $A$ (up to scaling) as its Reeb field and $\varphi$ a vacuum up
to a conformal transformation. More generalizations have been also studied.


\end{abstract}


\maketitle

\section{Introduction}
The zero mode equation 
\eq{\label{eq01}
\sigma\cdot (-i \nabla -A)\varphi =0 \quad \hbox{ in }\R^n
}
 plays an important  role in physics and mathematics. The existence of its  non-trivial solution $(A, \varphi)$ with $\varphi\not \equiv0$, which is called a zero mode,  is crucial for the stability of Coulomb systems with magnetic fields \cite{Froelich-Lieb-Loss}. The first explicit examples were found by Loss-Yau in \cite{Loss_Yau_86} in the 3-dimensional case, which were generalized to the $n$-dimensional case by Dunne-Min \cite{Dunne_Min_08}. Since then there
have been amount of existence results, though in general the existence is rare, see \cites{FL1, Aharonov-Casher19, Balinsky-Evans01, Erdoes-Solovej01, Benguria15, Borrelli21, Elton02, Saito08, Ross-Schroers18, Grosse23}. 

Recently Frank-Loss \cite{Frank_Loss_2024} gave a sharp criterion on its existence of zero modes (with suitable regularity)  by establishing the following sharp inequality
\eq{\label{eq02}
\norm{A}_{L^n}^2 \ge \frac n{n-2} S_n  \quad\hbox{in}\ \mathbb{R}^n
}
with equality if and only if $n$ is odd and $(\varphi,A)$ are solutions given in \cites{Loss_Yau_86, Dunne_Min_08}. Here $S_n$ is the optimal Sobolev constant.
Their result was generalized very recently by J. Reu\ss{} in \cite{Reuss25} to general spin manifolds, on which the zero mode equation \eqref{eq01} is naturally generalized to 
\eq{\label{eq03}
\D \varphi = iA \cdot \varphi,
}
where $\varphi$ is a (complex valued) spinor field, $\cdot$ is the Clifford multiplication, $\D$ is the Dirac operator and $A$ is a vector field on a spin manifold $M$. See a brief introduction in the next section. He proved the following optimal inequality 
\eq{\label{eq04}
\norm{A}_{L^n}^2 \ge \frac {n}{4(n-1)} Y(M,[g]),
} where $Y(M,[g])$ is the Yamabe constant of
$(M,g)$. Equality of \eqref{eq04} is also addressed in \cite{Reuss25}.

The objective of this paper is twofold. We first give a simpler proof of \eqref{eq04},  and then complete the discussion of the equality case.
The main theorem of this paper states
\begin{theorem}
    \label{main_thm}
Let $(M^n,g)$ $(n\geq3)$ be a closed manifold with a spin structure $\sigma$ and let
$(\varphi,A)$ be a non-trivial solution to \eqref{eq03} with $\varphi\in L^p$  $(p>\frac{n}{n-1})$.
Then
\eq{\label{INQ}
\norm{A}_{L^n}^2 \ge \frac {n}{4(n-1)} Y(M,[g]),
}
and equality holds if and only if 
$\varphi$ is a non-parallel real Killing spinor up to a conformal transformation, and if and only if 
$(M,g)$ is a Sasaki-Einstein manifold with $A$ (up to scaling) as its Reeb field  and $\varphi$  a vacuum up to a conformal transformation.  Here  the definition of vacuum is given in Definition \ref{def_vacuum} below.
\end{theorem} 

In contract to many optimal inequalities for scalar functions, for example the Sobolev inequality and the logarithmic Sobolev inequality, there are only few optimal inequalities for vector valued (or spinor valued) functions. Inequality \ref{eq02} (or \eqref{eq04}) is one of them, together with the Hijazi inequality \cite{H86} and its consequence, the spinorial Sobolev inequality, see \cites{A03, WZ25}. Most of powerful tools for establishing inequalities for scalar functions, for example the rearrangement method and the mass transport method, do not work for vector valued functions.
The main difficulty in the proof of \eqref{eq02} and \eqref{eq04} is the existence of zeros of $\varphi$. If $\varphi$ is nowhere vanishing, inequality \eqref{eq02} follows directly from the Schr\"odinger-Lichnerowicz formula, H\"older's inequality and the optimal Sobolev inequality (See \cite{FL22}). In order to handle this problem, Frank-Loss  used a regularization 
$\abs{\varphi}_\epsilon=\sqrt{\abs{\varphi}^2+\epsilon^2}$ for $\epsilon>0$. After much effort they succeeded to prove inequality \eqref{eq02} by letting $\epsilon\to 0$. Due to the limiting argument, equality  becomes more difficult to deal with. Nevertheless they managed to  prove that $(\varphi,A)$ achieves equality if and only if it is the solution found in \cites{Loss_Yau_86, Dunne_Min_08}. For a general spin manifold, J. Reu\ss{} followed the same strategy to show inequality \eqref{eq04} and proved that equality implies that up to a conformal transformation $(M,g)$ carries a real Killing spinor. And hence the universal covering of  $(M, g)$ must be one of manifolds in the classification of C. B\"ar \cite{Bar1993}, see also \cite{McWang1989}.

Another method to deal with the difficulty of the appearance  of zeros in $\varphi$ was used by  Borrelli-Malchiodi-Wu in \cite{Borrelli21}, where they managed to prove that the zero set is sufficiently small and can be ignored in the application of the Schr\"odinger-Lichnerowicz formula.

Our proof  is inspired by the proof of the Hijazi inequality  in \cite{H86} and the work of J. Julio-Batalla \cite{Jurgen_Julio-Batalla}, as well as Frank-Loss \cite{Frank_Loss_2024}.
The main idea is that we reduce the problem involving the spinor field $\varphi$ in \eqref{eq03}  to a problem involving scalar functions, which is related to a suitable eigenvalue problem; meanwhile we exploit the conformal invariance of the zero mode equation \eqref{eq03}.
Via this method, the zero set of $\varphi$ does not arise any problem.

We actually prove a stronger form of \eqref{eq04} by inserting a new conformal invariant $Y([g,A])$ into the inequality:
\eq{\label{eq05}
\norm{A}_{L^n}^2 \ge \frac {n}{4(n-1)} Y([g,A]) \ge \frac {n}{4(n-1)} Y(M,[g]).
}
Here $Y([g,A])$ is a conformal invariant defined by
\eq{\label{eq06}
Y([g,A]):=\|A\|_{L^n}^{2}
\inf \bigg\{ 
\frac {\int_M uL_gu \,\rd V_g}{
\int_M \abs{A}_g^2 u^2 \,\rd V_g}  \,\bigg |\, u \in W^{1,2} \hbox{ with } \int_M \abs{A}_g^2 u^2 \rd V_g > 0 \bigg \}}
and $Y(M,[g])$ is the classical Yamabe constant defined by 

 \eq{
 Y(M,[g]) =\inf _{0\not\equiv u\in W^{1,2} } \frac {\int_M uL_g u \,\rdV_g}{(\int _M u^{\frac {2n}{n-2}}\rdV_g)^{\frac{n-2}{n}}}.
 }
Moreover, $[g,A]$ is the conformal class of $(g,A)$ defined in Definition \ref{Def2} below and
\eq{\label{eq_Yamabe}
L_g u \coloneqq -\frac{4(n-1)}{n-2}\Delta_g u+{\rm R}_g u}
is the conformal Laplacian. The second inequality follows simply from H\"older's inequality, while
the first inequality in \eqref{eq05}, which is slightly stronger than \eqref{INQ}, 
follows from the following variant of the Schr\"odiner-Lichnerowicz formula
    \eq{\label{Sch-L}
    \int_M u^{-\frac 2 {n-2}}\left(\frac{4(n-1)}n|\D_g \varphi|^2 -u^{-1}L_gu|\varphi|^2\right)\rdV_g=\int_M u^{\frac {2(n-1)}{n-2}} |P_g(u^{-\frac n {n-2}}\varphi)|^2 \rdV_g
    }
    where $P_g$ is the twistor operator, at least if $\varphi$ and $u$ sufficiently regular. For the proof of \eqref{Sch-L}, see for instance \cite{Herzlich-M}. 
    In fact,  if we insert the zero mode $\varphi$ and the first eigenfunction $u$ of the above problem \eqref{eq06} into \eqref{Sch-L}, we will obtain 
    \eq{
    \int_M u^{-\frac 2{n-2}}|A|_g^2|\varphi|^2\left (\frac {4(n-1)}n  -\|A\|_{L^n}^{-2} Y([g, A]) \right) \rdV_g=\int_M u^{\frac {2(n-1)}{n-2}} |P_g(u^{-\frac n {n-2}}\varphi)|^2 \rdV_g\ge 0,
    } and hence
    $\|A\|_{L^n}^2 \ge \frac{n}{4(n-1)} Y([g,A])$.
    Using this approach one avoids the problem arising from the appearance of zeros in $\varphi$, provided that $\varphi$ and $u$ are regular enough. 
    However, in general  we only assume that $\varphi, u \in W^{1,2}$. This is not enough to apply  \eqref{Sch-L}.
   To deal with this problem, one may need to use an appropriate approximation method or a regularization method as in \cite{Frank_Loss_2024}. Here we use another method, which we believe, has its own interest. A similar idea was used in \cite{Jurgen_Julio-Batalla}.
    
To overcome this problem, we consider a deformation of $L_g$ by
\eq{
L^a_g u\coloneqq-a\frac{4(n-1)}{n-2} \Delta_g u+{\rm R}_gu \quad\hbox{for}\quad a\in (0,1],
}
and the corresponding eigenvalue problem
\eq{
\mu_a(g, A)\coloneqq \|A\|_{L^n} ^2
\inf_{0\not \equiv u \in W^{1,2}}\frac {\int (a \frac {4(n-1)}{n-2}\abs{\nabla u}^2 _g+{\rm R}_g u^2 )\rdV_g}{
\int \abs{A}_g^2 u^2\rdV_g}.
}
Define a conformal invariant
\eq{
\gamma_a([g,A]) \coloneqq \sup_{(\tilde g, A_{\tilde g})\in [g,A]} \mu_a (\tilde g, A_{\tilde g}).
}
It is crucial to observe that   $$\gamma_a ([g,A]) =Y([g,A]), \quad \hbox{  for any } a\in (0,1].$$ Then  using only the (ordinary) Schr\"odinger-Lichnerowicz formula \eq{\label{Sch-Lich}
0 = -\frac{n-1}{n}\int \abs{\D\varphi}^2 + \frac{1}{4}\int {\rm R}\abs{\varphi}^2 + \int \abs{P\varphi}^2
}
we  prove $\|A\|_{L^n} ^2 \ge \frac n {4(n-1)} \gamma_a ([g,A])$ for a small $a>0$,
and hence 
$\|A\|_{L^n} ^2 \ge \frac n {4(n-1)} Y([g,A]).$

Now if $(\varphi,A)$ achieves equality, we in fact have two equalities, i.e.
\eq{\label{eq09}
\norm{A}_{L^n}^2 = \frac {n}{4(n-1)} Y([g,A]) = \frac {n}{4(n-1)} Y(M,[g]).
}
Moreover, it is ready to show that $Y([g,A])$ is achieved by a conformal metric $\tilde g$. Then it follows easily that $\varphi$ is a twistor spinor with constant length, up to a conformal transformation. This was also obtained in
\cite{Reuss25} by using a similar method of \cite{Frank_Loss_2024} mentioned above. In principle the first equality  implies that $\varphi$ is a twistor, while the second implies that the length of $A$, and hence the length of $\varphi$, is constant, up to a conformal transformation.

It is a well-known result of Lichnerowicz \cite{Lichnerowicz87} that a twistor spinor $\varphi$ with constant length can be decomposed as a sum of two real Killing spinors $ \varphi=\varphi_++\varphi_-$, where $\varphi_\pm$ is a $\pm b$-Killing spinor for some $b\in \R$.
One main step to find an equivalent condition for equality  is now to  prove further that $\varphi$ itself is in fact a real Killing spinor, i.e., one of $\varphi_\pm$ vanishes.
If using the classification of manifolds with real Killing spinors proved by C. B\"ar \cite{Bar1993},  we know that the universal cover of $(M,g)$ is a Sasaki-Einstein manifold, together with two non-Sasakian exceptions:  $M$ is a 6-dimensional nearly K\"akler non-K\"ahler manifold of constant type one, or $M$ is a 7-dimensional nearly parallel $G_2$
manifold. Hence it remains to exclude these two exceptional cases by using extra information from $A$.
In this paper, we use a more direct method, without using the classification and without the assumption of  simply connectedness. Inspired by Friedrich-Kath \cite{Friedrich-Kath89_JDG}, we show that $n$ must be odd and $(M,g)$ must be a Sasaki-Einstein manifold. Moreover, we proved that both $\varphi_+$ and $\varphi_-$ are also solutions to the zero mode equation \eqref{eq03} with the same $A$, which implies that they are vacuums in the sense of Definition \ref{def_vacuum} below. The vacuum is unique, following an idea in \cite{FL1}. Therefore one of them must be zero and hence $\varphi$ itself is a real Killing spinor and a vacuum. The converse is not difficult to prove.

The rigidity part of Theorem \ref{main_thm} answers the question that when a zero mode $(\varphi,A)$ achieves equality. It is different to ask which manifolds admit a  zero mode achieving equality.
If $(M,g)$ is a simply connected Sasaki-Einstein manifold, by a result of Friedrich-Kath \cite{Friedrich-Kath} there exists a real Killing spinor solving the zero mode equation \eqref{eq03} with $A$ (up to scaling) as its Reeb field. Hence every simply connected Sasaki-Einstein (which is automatically spin) admits a zero mode achieving equality. For a general Sasaki-Einstein manifold,
since $\pi_1(M)$ is finite by Myers' theorem, we can always lift $(M, g)$  to its universal cover, to which the above result can be applied, namely there is a real Killing spinor solving \eqref{eq03}. However it may not descend  to $M$. Examples are given by quotients of $\SS^5$:
among quotients $\SS^5 \slash {\mathbb Z}_q$, where $\mathbb{Z}_q$ acts by sending $(z_1, z_2, z_3)  \to (\xi z_1, \xi z_ 2, \xi z_3)$ with $\xi$ a primitive $q$-th root of unity,  only $\SS^5$  and $\SS^5\slash \mathbb{Z}_3$ admit Killing spinors, see \cite{Sulanke80} and also \cite{Friedrich-Kath89_JDG}. Therefore on other quotients, though they are Sasaki-Einstein manifolds, there is no zero mode achieving equality.
To completely answer the second question one needs to well understand non-simply-connected Sasaki-Einstein manifolds.
We remark that $\SS^5\slash \mathbb{Z}_2$ is even not a spin manifold, while any simply connected Sasaki-Einstein manifold is spin. For Sasakian geometry and Sasaki-Einstein manifolds  we refer to \cite{Boyer_G_book}.

We will also show that the conclusion holds true for $2$-dimensional case as well. In this case the Yamabe constant is equal to $4\pi\chi(\Sigma)$ and equality never occurs.

\begin{theorem}
    \label{main_thm_2_dim}
Let $(\Sigma^2,g)$ be a closed Riemann surface with a spin structure $\sigma$ and let
$(\varphi,A)$ be a non-trivial solution to \eqref{eq03} with $\varphi\in L^2$ and
$A\in L^q$ $(q>2)$.
Then
\eq{
\norm{A}_{L^2}^2 > 2\pi\chi(\Sigma).
}
\end{theorem}

This is a B\"ar type inequality (see \cite{Bar1992}), while Theorem \ref{main_thm} is a Hijazi type result (see \cite{H86}). See also the recent work of Julio-Batalla  \cite{Jurgen_Julio-Batalla}.

Our approach is more flexible and can be used to deal with more general cases, which will be studied in Section \ref{sec5} below. 

Concerning the regularity of $\psi$, the most interesting case is $\psi\in L^{\frac {2n}{n-1}}$. Together with $A\in L^{n}$, \eqref{eq03} implies $\D\psi \in L^{\frac {2n}{n+1}}$.  It is remarkable that these three norms are conformally invariant. Hence under the assumption of  boundedness of these norms, the corresponding problems on $\R^n$, $\mathbb{H}^n$ and $\SS^n$
are in fact equivalent, thanks to the conformal invariance of \eqref{eq03}. Analytically it is interesting to consider a more general case: $\varphi \in L^p$  for some $p> n\slash (n-1)$. It was proved in \cites{FL1, Frank_Loss_2024} that the assumption $\varphi \in L^p$ with $p> n\slash (n-1)$ actually implies $\varphi \in L^r$ for any  $r> n\slash (n-1)$ for $M=\R^n$. In the Appendix we mimic their proof for a general $M$. The case $p=n\slash (n-1)$ is a critical case. It is a very challenging problem that if  any solution $\varphi$  of the zero mode equation with  $\varphi \in L^{ n\slash (n-1)}$ and $A\in L^n$ satisfies $\varphi \in L^p$ for any $p\ge n\slash (n-1)$. Very recently, De Lio-Riviere-Wettstein  \cite{Riviere24} gave an affirmative answer for lower dimensions, with a slightly different condition.

The study of  zero modes attracts more and more attention of mathematicians. 
Equation \eqref{eq03} itself  relates  a spinor field $\varphi$ and a vector field $A$ (or its dual, a one form). 
It is interesting to see that inequality 
\eqref{eq04} provides an optimal inequality relating a conformal norm ($L^n$-norm) for $A$ and a conformal invariant, the Yamabe invariant, which is actually the best Sobolev constant for a given conformal class $(M,[g])$,
and opens many interesting questions in \cites{FL1, FL22, Frank_Loss_2024} concerning optimal inequalities of Sobolev type for vectorial, spinorial and form valued  functions. See also our recent work \cites{WZ25, WZ25c}, motivated by \cites{FL1, FL22, Frank_Loss_2024} .

\

\noindent{\it Organization of the rest of the paper}. In Section 2 we review briefly spin geometry, the Dirac operator, and the zero mode equation together with its conformal invariance. A family of conformal invariants is introduced in Section 3, with which we provide the proof of inequality \eqref{eq04} and a part of the equality case. Its rest part is given in Section 4, together with a discussion on the relation between zero modes and Sasaki-Einstein manifolds.
Section 5 devotes to a more general equation.  We ask questions for a related conformal invariant in Section 6. In the Appendix we provide the proof for the regularity for the completeness.

\section{The zero mode equation}\label{Section_the_zero_mode_equation}

\subsection{Spin manifolds and the Dirac operator}

In this subsection we recall some basics about spin manifolds and the Dirac operator. For general information about spin geometry and the Dirac operator, we refer to \cites{Lawson, Baum_Book_90, Friedrich_Book, Ginoux09}.

Let $(M,g)$ be an orientable Riemannian manifold of dimension $n\geq 2$. Over $M$ one can define an ${\rm SO}(n)$-principle bundle $P_{{\rm SO}(n)}M$ with fibres being oriented orthonormal bases. We call $M$ a spin manifold if $P_{{\rm SO}(n)}M$ can be two-fold lifted up to $P_{{\rm Spin}(n)}M$, where the Lie group ${\rm Spin}(n)$ is the simply connected two-fold cover of ${\rm SO}(n)$. The cover $\sigma: P_{{\rm Spin}(n)}M \ra P_{{\rm SO}(n)}M$ is called a spin structure. We only consider spin manifolds in this paper. It is well known that $M$ is spin if and only if the second Stiefel-Whitney class of $M$ vanishes. In particular, $\S^n$ is a spin manifold.

We denote by $\Sigma M$ the associated complex vector bundle of the principle bundle $P_{{\rm Spin}(n)}M$, which has complex rank $2^{[\frac n2]}$. The Riemannian metric $g$ on $M$ endows a canonical Hermitian metric and the associated spin connection on $\Sigma M$.
We denote by $\<\cdot,\cdot\>$ the Hermitian metric and by $\nabla$ the spin connection
The same notations are also used for vector fields, with the meaning clear from context.
For any $L^2$-self-adjoint operator $T$, for instance $T=\D, \nabla^*\nabla, \Delta$, etc., note that $\int\<T\psi,\psi\>=\int\rRe\<T\psi,\psi\>$. In the sequel, we omit to emphasize $\rRe$ if not necessary.

A section of $\Sigma M$ is called a \textit{spinor field}, often denoted by $\psi,\varphi$, etc. Vector fields act on spinor fields by $\gamma: TM \ra {\rm End}_{\mathbb{C}}(\Sigma M)$. For short we use the notation $X\cdot\psi\coloneqq \gamma(X)(\psi)$. The action is anti-symmetric with respect to $\<\cdot,\cdot\>$ and obeys the so-called Clifford multiplication rule $X\cdot Y\cdot \psi + Y\cdot X\cdot \psi = -2 g(X,Y)\psi$.

Let $\{e_{j}\}_{j=1}^{n}$ be an orthonormal frame of $M$. The Dirac operator $\D:\Gamma(\Sigma M) \ra \Gamma(\Sigma M)$ is locally defined by
\eq{
    \D\psi \coloneqq \sum_{j=1}^{n} e_{j}\cdot \nabla_{e_{j}}\psi, \quad\forall \,\psi\in\Gamma(\Sigma M).
}
It is well known that $\D$ is a first-order elliptic self-adjoint operator, which plays the role as ``square root'' of Laplacian through the famous Schr\"odinger-Lichnerowicz formula 
\eq{\label{Bochner}
    \D^2 = -\Delta + \frac{{\rm R}}{4},
}
where ${\rm R}$ is the scalar curvature. A class of special spinor fields, \textit{Killing spinors}, is defined by the following equation
\eq{
    \nabla_{X}\psi = b X\cdot \psi, \quad \forall \,X\in\Gamma(TM),
}
where $b\in\mathbb{C}$ is a constant and is called the Killing number.
In particular, $\psi$ is called a real Killing spinor if $b\in\mathbb{R}$.
Existence of a non-trivial Killing spinor is a demanding requirement of manifold.
Complete simply connected Riemannian spin manifolds  carrying non-zero space of $b$-Killing spinors have been classified by M. Wang \cite{McWang1989} for $b=0$, C. B\"ar \cite{Bar1993} for $b\in\mathbb{R}\backslash\{0\}$ and H. Baum, etc. \cite{Baum_Book_90} for $b\in i\mathbb{R}\backslash\{0\}$, respectively.

Since in the paper the Killing spinors are all non-parallel real Killing spinor, for convenience we always assume $b=\pm \frac 12 $, if no confusion.

\subsection{The zero mode equation}\label{sec2.2}

 On a closed spin manifold $(M^n,g)$ $(n\geq2)$ we consider the following zero mode  equation
 \eq{\label{eq1.0}
 \D \varphi=i A\cdot \varphi,}
 where $i$ is the imaginary unit.
We start with the following fact.

\begin{lemma}[Gauge invariance]
\label{lem2.1} Equation \eqref{eq1.0} is gauge invariant, 
i.e., it is invariant under transformations
\[
(A, \varphi) \mapsto  (A+\nabla f, e^{if}\varphi) \quad \hbox{for any function} \ f.
\] \end{lemma}
\begin{proof}
In fact, using the following formula
\eq{
\D(h\varphi) = h\D\varphi + \rd h \cdot\varphi \quad \hbox{for any function} \ h \ \hbox{and any spinor field} \ \varphi,
}
we have
\eq{
\D (e^{if}\varphi) = e^{if} i A\cdot \varphi + e^{if} i \nabla f \cdot \varphi = i(A +\nabla f) \cdot (e^{if} \varphi).
}

\end{proof}
From Lemma \ref{lem2.1} one can see easily that \eqref{eq04}
is equivalent to
\eq{
\inf _f \norm{A-\nabla f}_{L^n}^2 \ge \frac {n}{4(n-1)} Y(M,[g]).
}

 Since we are interested in finding a lower bound of $\int \abs{A}^n$, without loss of generality, we may assume that $
 \int \abs{A}^n<\infty.$ 
 It is convenient to consider the following equivalent problem
 \eq{\label{eq1.1}
 \D \varphi = i\lambda A\cdot \varphi}
 for a vector field $A\in L^n$ with $\int \abs{A}^n =1$, $\lambda>0$ and $\varphi\not\equiv0$. Then the desired inequality  can be written as
 \eq{\label{ineq0}  \lambda^2 \ge\frac n {4(n-1)} Y(M, [g]).}

     We recall important conformal properties of the Dirac operator.
     Let $\tilde{g}=h^{\frac{4}{n-2}}g$. The isometry $(TM,g) \to (TM, \tilde{g})$ with $X\mapsto \tilde{X}=h^{-\frac{2}{n-2}}X$ induces a canonical isometry $F:(\Sigma M,g) \to (\Sigma M, \tilde{g})$ with $\varphi \mapsto \tilde{\varphi}$. It preserves the Hermitian metric on spinor bundle and the induced spin representation $\tilde{\cdot}$ satisfies $\widetilde{X\cdot\varphi} = \tilde{X}\,\tilde{\cdot}\,\tilde{\varphi}$. The induced Dirac operator $\D_{\tilde{g}}$ satisfies the well-known property (see for instance \cite{Ginoux09}*{Proposition 1.3.10})
     \eq{\label{conf_trans}
     \D_{\tilde{g}}(h^{-\frac{n-1}{n-2}}\tilde{\varphi}) = h^{-\frac{n+1}{n-2}}\widetilde{\D_g\varphi}.
     }

\begin{definition}\label{Def2}
The conformal class of $(g, A)$ is defined by
\eq{
[g,A] \coloneqq \{(\tilde g, A_{\tilde g}) \,:\, \tilde g=h^{\frac 4 {n-2}}g,\ A_{\tilde g}=h^{-\frac 4 {n-2}}A \ \hbox{for some positive function}\ h\}.
}

\end{definition}   
\begin{lemma}[Conformal invariance]\label{conformal_invariance}
     Equation \eqref{eq1.1} is conformally invariant in the following sense: Let $(\varphi,A)$ be a solution to \eqref{eq1.1} w.r.t. metric $g$ and
     $\tilde g =h^{\frac 4{n-2}} g$.
     Define \eq{\varphi_{\tilde{g}} \coloneqq h^{-\frac{n-1}{n-2}}\tilde{\varphi}, \qquad A_{\tilde{g}} \coloneqq h^{-\frac{2}{n-2}}\tilde{A} = h^{-\frac{4}{n-2}}A.} Then $(\varphi_{\tilde g},A_{\tilde g})$ is a solution to \eqref{eq1.1} w.r.t. metric $\tilde{g}$, i.e.
     \eq{
\D_{\tilde{g}}\varphi_{\tilde{g}} = i\lambda  A_{\tilde{g}}\,\tilde\cdot\,\varphi_{\tilde{g}},
     }
     with the same $\lambda$.
     \end{lemma} 
     
     \begin{proof}
    In fact,  \eqref{conf_trans} and \eqref{eq1.1} imply
    \eq{
    \D_{\tilde{g}}\varphi_{\tilde{g}} &= h^{-\frac{n+1}{n-2}}\widetilde{\D_g\varphi} = i\lambda h^{-\frac{n+1}{n-2}}\widetilde{A\cdot\varphi} = i\lambda h^{-\frac{n+1}{n-2}}\tilde{A}\,\tilde{\cdot}\,\tilde{\varphi}\\
    &= i\lambda h^{-\frac{n+1}{n-2}}\cdot h^{\frac{2}{n-2}}A_{\tilde{g}}\,\tilde{\cdot}\,h^{\frac{n-1}{n-2}}\varphi_{\tilde{g}} = i\lambda A_{\tilde{g}}\,\tilde{\cdot}\,\varphi_{\tilde{g}}.
    }
    Moreover, we have
     \eq{\label{A}
   \abs{A_{\tilde{g}}}_{\tilde g}  =h^{\frac{2}{n-2}} \abs{h^{-\frac{4}{n-2}}A}_g   =h^{-\frac {2}{n-2}}\abs{A}_g , 
    } and hence
    \eq{
    \int \abs{A_{\tilde{g}}}^n_{\tilde g}  \rd V_{\tilde{g}} = \int \big( h^{-\frac{2}{n-2}} \abs{A}_g \big)^n \cdot h^{\frac{2n}{n-2}}\, \rd V_g = \int \abs{A}_g^n \rd V_g = 1.
    } 
    \end{proof}
Therefore, if one wants to estimate $\lambda$ like in this paper, one can choice a suitable conformal metric. 
\subsection{The Yamabe invariant}
 The celebrated Yamabe constant is defined by 
\eq{\label{Yamabe_constant}
 Y(M,[g]) =\inf _{0\not\equiv u\in W^{1,2} } \frac {\int_M uL_g u \,\rdV_g}{(\int _M u^{\frac {2n}{n-2}}\rdV_g)^{\frac{n-2}{n}}},
 }
where $L_g$ is the conformal Laplacian given in \eqref{eq_Yamabe}. A conformal metric $\tilde g= u^{\frac {4}{n-2}}g$ 
has scalar curvature
\eq{\label{scalar_curvature}
R_{\tilde g}=u^{-\frac {n+2}{n-2}} L_g u  = u^{-\frac {n+2}{n-2}} \left\{-\frac{4(n-1)}{n-2}\Delta_g u+{\rm R}_g u\right\}.
}
By the resolution of the  Yamabe  problem through Yamabe, Trudinger, Aubin and Schoen, the infimum  in \eqref{Yamabe_constant} is always achieved by a conformal metric $\tilde g$, which is usually called a Yamabe minimizer. For the Yamabe problem, see a nice Survey \cite{Lee_Parker}. If $(M,g)$ is Einstein, by Obata \cite{Obata71}, other conformal metrics $\tilde g$ of constant scalar curvature are also Einstein. If in addition $(M,g)$ is not conformally equivalent to $\SS^n$, then $\tilde g =cg$ for some constant $c>0$ and hence  any conformal metric of constant scalar curvature is a Yamabe minimizer.
It is easy to check that the Yamabe constant of $\SS^n$ is
\eq{
Y(\SS^n)=\frac{4(n-1)}{n-2} S_n = n(n-1)\omega_n^{\frac{2}{n}}.
}

\section{A family of conformal invariants}
\label{sec3}

Let us introduce a family of eigenvalue problems, from which we then define a family of conformal invariants. It is interesting to observe that these invariants turn out to be identical, see Proposition \ref{cor7.9} below.

Let $a\in (0, 1]$. We consider the following operator
\eq{
L^a_g \coloneqq -a \frac {4(n-1)}{n-2}\Delta_g + {\rm R}_g
}
and its eigenvalue problem w.r.t. the weight $\abs{A}_g^2$, i.e.
\eq{\label{def_mu_a}
\mu_a(g, A)\coloneqq
\inf \bigg\{ \frac {\int (a \frac {4(n-1)}{n-2}\abs{\nabla u}^2 _g+{\rm R}_g u^2 )\rd V_g}{
\int \abs{A}_g^2 u^2\rd V_g}  \,\bigg |\, u \in W^{1,2} \hbox{ with } \int_M \abs{A}_g^2 u^2 \rd V_g > 0 \bigg \},}
which implies
\eq{\label{eig}
{\int (a \frac {4(n-1)}{n-2}\abs{\nabla u}^2 _g+{\rm R}_g u^2 )\rd V_g} \ge \mu_a (g, A) 
\int \abs{A}_g^2 u^2\rd V_g, \quad \forall \, u\in W^{1,2}.
}
The first positive eigenfunction $u\in W^{1,2}$ satisfies
\eq{\label{eigen_problem}
L_g^a u =-a\Delta_g u+R_g u= \mu_a(g,A)\abs{A}_g^2 u.
}

\begin{definition}\label{def_gamma_a}
For any given $a\in (0, 1]$, we define a conformal invariant for the class $[g,A]$
\eq{
\gamma_a([g,A]) \coloneqq \sup_{(\tilde g, A_{\tilde g})\in [g,A]} \mu_a (\tilde g, A_{\tilde g}).
}
    
\end{definition}

We will prove that  all $\gamma_a$'s are actually the same.
First of all, we observe that the case $a=1$ is special, since $L_g:=L^{1}_g$ is the conformal Laplacian, or the conformal Yamabe operator.
It is well-known that for $\tilde g =h^{\frac 4 {n-2}}$ we have
\eq{
L_{\tilde{g}}u = h^{-\frac{n+2}{n-2}}L_g(hu), \qquad
{\rm R}_{\tilde{g}} = h^{-\frac{n+2}{n-2}}L_gh.
}
Set
\eq{
I_{g, A}(u) \coloneqq \frac{\int u \, L_gu \,\rd V_g}{\int |A|_g^2 \, u^2 \,\rd V_g}.
}

\begin{definition} \label{def_I_gA}
We define $Y(g,A) \coloneqq \mu_{1}(g,A)$, i.e.,
\eq{
Y(g, A) \coloneqq \inf_{0\not\equiv u\in W^{1,2}} I_{g, A} (u).
}
    
\end{definition}
It is clear that $Y(g, A)$ is the first eigenvalue of the conformal Laplacian $L_g$ with respect to the weight $|A|^2_g$ and is achieved by a positive function $u_0$.

\begin{lemma}\label{conformal_1} $I_{g,A}$ is conformal invariant in the sense that
    $I_{\tilde{g},A_{_{\tilde{g}}}}(u) = I_{g,A}(hu)$ for any $\tilde g =h^{\frac 4{n-2}}g$.
\end{lemma}

\begin{proof}
    From the above discussion we have
    \eq{
    I_{\tilde{g},A_{_{\tilde{g}}}}(u) = \frac{\int u \, L_{\tilde{g}}u \,\rd V_{\tilde{g}}}{\int |A_{\tilde{g}}|_{\tilde{g}}^2 \, u^2 \,\rd V_{\tilde{g}}} = \frac{\int h^{-1}(hu) \, h^{-\frac{n+2}{n-2}} L_{g}(hu) h^{\frac{2n}{n-2}} \,\rd V_{g}}{\int h^{-\frac{4}{n-2}} \abs{A}_{g}^2 \, h^{-2} (hu)^2 h^{\frac{2n}{n-2}} \,\rd V_{g}} = I_{g,A}(hu).
    }
\end{proof}

Lemma \ref{conformal_1} has a direct consequence, which shows that $Y(g,A)$ is a conformal invariant in the following sense.
\begin{corollary}\label{coro7.6}
    $Y(\tilde{g},A_{\tilde{g}}) = Y(g,A)=\mu_{1}(g,A)$.
\end{corollary}

Therefore, we denote this conformal invariant by $Y([g,A])$, which is the same  one as in the Introduction.
In contrast, $\mu_a (g,A)$ $(0<a<1)$ is not conformally invariant. Nevertheless, its supreme over $[g,A]$, $\gamma_a([g,A])$, has the following property.

\begin{proposition}\label{prop7.7} For any $a\in (0, 1] $,
    \eq{ \gamma_a([g, A] ) = Y([g,A]).} 
    Moreover, $\gamma_a ([g,A])$ is achieved by a conformal metric $\tilde g \in [g]$.
\end{proposition}

\begin{proof}
    Given any $0<a\leq 1  $, we have
    \eq{\label{add_eq1}
    Y([g,A]) = \inf_{u\not\equiv0} \frac{\int u \, L_gu \,\rd V_g}{\int |A|_g^2 \, u^2 \,\rd V_g} \geq \inf_{u\not\equiv0} \frac{\int \big( a \frac {4(n-1)}{n-2}\abs{\nabla u}_g^2 + {\rm R}_g u^2 \big) \,\rd V_g}{\int |A|_g^2 \, u^2 \,\rd V_g} = \mu_a(g,A).
    }
   In view of Corollary \ref{coro7.6}, taking supremum of the previous inequality over $[g,A]$ we have
    \eq{\label{add_eq2}
    Y([g,A]) \geq \sup_{\tilde{g}\in [g]} \mu_a(\tilde{g},A_{\tilde{g}})=\gamma_a([g,A]).
    }
    In order to prove the reverse inequality, we use the fact that the infimum in Definition \ref{def_I_gA} is achieved by a positive function $u_0$, which satisfies the corresponding Euler-Lagrange equation
    \eq{
    L_gu_0 = Y([g,A]) \abs{A}_g^2 u_0.
    }
     Together with  \eqref{scalar_curvature}  it implies for $\tilde{g} \coloneqq u_0^{\frac{4}{n-2}}g$,
     \eq{\label{constant_scalar_curvature}
    {\rm R}_{\tilde{g}} = u_0^{-\frac{n+2}{n-2}}L_gu_0 = u_0^{-\frac{n+2}{n-2}}Y([g,A]) \abs{A}_g^2 u_0 = Y([g,A]) \abs{A_{\tilde{g}}}_{\tilde{g}}^2,
    } where we have used \eqref{A}.
    Hence in view of \eqref{def_mu_a} and Definition \ref{def_gamma_a}, we have 
    \eq{\label{pf_prop7.7}
  \gamma_a([g,A])\ge   \mu_a(\tilde{g},A_{\tilde{g}}) \geq   \inf_{u\not\equiv0}\frac{\int {\rm R}_{\tilde{g}} u^2 \,\rd V_{\tilde{g}}}{\int |A_{\tilde{g}}|_{\tilde{g}}^2 \, u^2 \,\rd V_{\tilde{g}}} =  Y([g,A]).
    }
   \eqref{add_eq2} and \eqref{pf_prop7.7} imply $\gamma_a([g, A] ) = Y([g,A])$. Moreover \eqref{pf_prop7.7} implies that $\gamma_a([g, A])$ is achieved by $\tilde g =u_0^{\frac 4 {n-2} }g$.
  
\end{proof}

\begin{remark}\label{rmk7.8}
    From the proof we see that the supreme in Definition \ref{def_gamma_a}
    is achieved by $\tilde{g}=u_0^{\frac{4}{n-2}}g\in [g]$, where $u_0$ achieves the infimum in Definition \ref{def_I_gA}. Thus Lemma \ref{conformal_1} implies that $u=u_0^{-1}u_0=1$ achieves the infimum in the definition of $Y(\tilde{g},A_{\tilde{g}})$.
    Therefore  equality in \eqref{Holder_prop3.7} for $\tilde{g}$ implies $\abs{A_{\tilde{g}}}_{\tilde{g}}\equiv{\rm const}$.
\end{remark}

Finally we relate it to the classical Yamabe constant.

\begin{proposition}
\label{cor7.9} 
For any $a \in (0, 1]$,
    \eq{\gamma_a ([g,A])=Y([g,A])  \geq Y(M,[g]).}
\end{proposition}

\begin{proof}
  Let $u_0$ be the positive function achieving $Y([g,A])$.   By H\"older's inequality we have
    \eq{\label{Holder_prop3.7}
    Y([g,A]) =  \frac{\int u_0 \, L_g u_0 \,\rd V_g}{\int |A|_g^2 \, u_0^2 \,\rd V_g} \geq  \frac{\int u_0 \, L_gu _0\,\rd V_g}{ \big( \int u_0^{\frac{2n}{n-2}} \rd V_g\big)^{\frac{n-2}{n}} \big( \int \abs{A}^n \rd V_g\big)^{\frac{2}{n}} } =  \frac{\int u_0 \, L_gu_0 \,\rd V_g}{ \big( \int u_0^{\frac{2n}{n-2}} \rd V_g\big)^{\frac{n-2}{n}} } \ge Y(M,[g]),
    }
    where we have used the normalization $\norm{A}_{L^n}=1$.
\end{proof}

Since \eqref{ineq0} makes sense only if $Y(M,[g])>0$,  from now on 
 we assume that $Y(M,[g])>0$. In this case, it is well-known that there is a conformal metric $\tilde{g}\in[g]$ with positive scalar curvature. Hence without loss of generality we assume that ${\rm R}_g>0$.

\begin{theorem}
\label{prop_gamma_a}
Let $(M^n,g)$ $(n\geq3)$ be a closed spin manifold. 
If $(M,g)$ admits a non-trivial solution $(\varphi,A)$ to \eqref{eq1.1}, where $\varphi\in L^p$ with $p>\frac{n}{n-1}$ and $A\in L^n$,
   then 
   \eq{
   \lambda^2 \ge \frac{n}{4(n-1)}Y([g,A]).
   }
 \end{theorem}

\begin{proof} Due to  conformal invariance,  we may assume that ${\rm R}:={\rm R}_g>0.$
First of all, Lemma \ref{L^r} implies that $\varphi\in L^{\frac {2n}{n-2}}$, which in turn implies that $A\cdot \varphi \in L^2 $ and $\varphi \in W^{1,2}$.
The Schr\"odinger-Lichnerowicz formula for spinor fields, \eqref{Bochner}, implies

\eq{
\frac{1}{2}\Delta\abs{\varphi}^2 = \<\Delta\varphi,\varphi\> + \abs{\nabla\varphi}^2 = \<-\D^2\varphi + \frac{{\rm R}}{4}\varphi,\varphi\> + \abs{\nabla\varphi}^2.
}
Integrating  it, together with  \eq{\label{eq1.4}
        \abs{\nabla\varphi}^2 = \frac{1}{n}\abs{\D\varphi}^2 + \abs{P\varphi}^2,
        }
 we have (i.e. \eqref{Sch-Lich})
\eq{\label{eq_b1}
0 = -\frac{n-1}{n}\int \abs{\D\varphi}^2 + \frac{1}{4}\int {\rm R}\abs{\varphi}^2 + \int \abs{P\varphi}^2,
} 
Here $P$ is the twistor operator (also called the Penrose operator) defined by $P_X\varphi \coloneqq \nabla_X\varphi + \frac{1}{n}X\cdot\varphi$.
It is easy to show that \eqref{eq_b1} is valid not only for smooth $\varphi$, but also for any $\varphi\in W^{1,2}$.
Putting the  zero mode equation into \eqref{eq_b1} gives
\eq{\label{reminder_twistor}
0= \int ({{\rm R}} - \frac{4(n-1)}{n}\lambda^2\abs{A}^2)\abs{\varphi}^2 +4 \int \abs{P\varphi}^2.
}
We now prove \eq{\lambda^2 \ge \frac{n}{4(n-1)}Y([g,A]) \label{eq3.9a}} by contradiction.
Assume that
$\lambda^2 \leq \frac{n}{4(n-1)}(Y([g,A]) - \delta)$
for some $\delta>0$.
For any $a\in (0,1]$ and the corresponding metric $\tilde{g}_a\in[g]$ that achieves the supreme in Definition \ref{def_gamma_a}, we have by Proposition \ref{prop7.7}
\eq{
\mu_a(\tilde{g}_a,A_{\tilde{g}_a}) = \gamma_a([g,A]) = Y([g,A]) \eqcolon Y.
}
For simplicity of notation, for now on  we omit the subscript $\tilde g_a$ and all terms are computed w.r.t $\tilde g_a$. By the definition of $\mu_a$ and \eqref{eig}
we have \eq{\label{eig2}
{\int (a \frac {4(n-1)}{n-2}\abs{\nabla |\varphi|}^2 +{\rm R} |\varphi|^2 )} \ge \mu_a (\tilde g_a, A_{\tilde g_a} ) {
\int \abs{A}^2 |\varphi|^2} =  Y \int \abs{A}^2 |\varphi|^2. 
}
Hence we have
\eq{
0 &\geq \int \left({{\rm R}} - (Y -\delta)\abs{A}^2\right)\abs{\varphi}^2 + 4\int \abs{P\varphi}^2\\
&= \int {{\rm R}}\abs{\varphi}^2 - (1-\delta Y^{-1})\int Y\abs{A}^2\abs{\varphi}^2 + 4\int \abs{P\varphi}^2\\
&\geq \int {{\rm R}}\abs{\varphi}^2 - (1-\delta Y ^{-1})\left(\int a\frac {4(n-1)}{n-2} \abs{\nabla\abs{\varphi}}^2 + {{\rm R}}\abs{\varphi}^2\right) +4 \int \abs{P\varphi}^2.
}
Using Kato's inequality  $|\nabla |\varphi|| \le |\nabla \varphi|$    and its consequence 
$\int |\nabla|\varphi||^2 \le \int |\nabla \varphi|^2$
we have
\eq{
0 \geq \int {{\rm R}}\abs{\varphi}^2 - (1-\delta Y^{-1})\left(\int a\frac {4(n-1)}{n-2}\abs{\nabla\varphi}^2 +{{\rm R}}\abs{\varphi}^2\right) +4 \int \abs{P\varphi}^2.
}
Combining \eqref{eq1.4} and \eqref{eq_b1} we have
\eq{
\int \abs{\nabla\varphi}^2 = \frac{1}{4(n-1)}\int {\rm R}\abs{\varphi}^2 + \frac{n}{n-1}\int \abs{P\varphi}^2.
}
Inserting it into the previous inequality yields
\eq{
0 &\geq \left(\delta Y^{-1} -
\frac{(1-\delta Y^{-1})a}{n-2}\right)
\int {\rm R}\abs{\varphi}^2 + \left(4-\frac{4n(1-\delta Y^{-1})a}{n-2}\right)\int \abs{P\varphi}^2.
}
Choosing $a>0$ small enough yields a contradiction.
\end{proof}

    \begin{remark}\label{rmk} The proof shows how  inequality \eqref{eq3.9a} follows from \eqref{reminder_twistor}. It works more generally, see Section \ref{sec5}.
    In fact, one can similarly prove that inequality \eqref{eq3.9a} holds true for any function $a\in L^n$: if there exists a non-trivial spinor field $\varphi\in L^{\frac {2n}{n+1}}$ such that
        \eq{
        |\D \varphi| \le |a| |\varphi|,
        }
        then
        \eq{\|a\|_{L^n}^2 \ge \frac n {4(n-1)} Y(M,[g]).}
    \end{remark}

Now we prove our first main result, which verifies the first ``if and only if" in Theorem \ref{main_thm}.

\begin{theorem}\label{thm7.12}
Let $(M^n,g)$ $(n\geq3)$ be a closed spin manifold. If $(M,g)$ admits a non-trivial solution $(\varphi,A)$ to \eqref{eq1.1}, where $\varphi\in L^p$ with $p>\frac{n}{n-1}$ and $A\in L^n$,
then
\eq{\label{ineq_thm7.12}
 \lambda^2 \geq \frac{n}{4(n-1)} Y(M,[g]).
}
Equality holds
if and only if 
$\varphi$ is a non-parallel real Killing spinor
up to a conformal transformation.
\end{theorem}

\begin{proof}
    By Theorem \ref{prop_gamma_a} and Proposition \ref{cor7.9} we have
    \eq{
    \lambda^2 \geq  \frac{n}{4(n-1)} Y([g,A]) \geq \frac{n}{4(n-1)} Y(M,[g]).
    }    
    Assume now equality holds, i.e.
    \eq{
    \lambda^2 = \frac{n}{4(n-1)} Y([g,A]) = \frac{n}{4(n-1)} Y(M,[g]).
    }
   Recall \eqref{constant_scalar_curvature} that for some $\tilde{g}\in[g]$ we have
    \eq{\label{constant_scalar_curvature_1}
    {\rm R}_{\tilde{g}} = Y([g,A])\abs{A_{\tilde{g}}}_{\tilde{g}}^2 = Y(M,[g])\abs{A_{\tilde{g}}}_{\tilde{g}}^2 = \frac{4(n-1)}{n}\lambda^2\abs{A_{\tilde{g}}}_{\tilde{g}}^2.
    }
    Thus \eqref{reminder_twistor} implies $P\varphi_{\tilde{g}}=0$, i.e., $\varphi_{\tilde{g}}$ is a twistor spinor. We recall the identity for any twistor spinor $\psi$ (see \cite{Baum_Book_90})
    \eq{\label{why_varphi_constant_lenth}
    \Delta \abs{\psi}^2 = \frac{{\rm R}}{2(n-1)}\abs{\psi}^2 - \frac{2}{n}\abs{\D\psi}^2.
    }
    Applying to $\psi=\varphi_{\tilde{g}}$  and to $\tilde g$ one can see from \eqref{constant_scalar_curvature} that $\Delta_{\tilde g} \abs{\varphi_{\tilde{g}}}^2=0$, and hence $\abs{\varphi_{\tilde{g}}}_{\tilde g}\equiv{\rm const}$.
    By a well known result (see for instance \cite{Ginoux09}*{Proposition A.2.1})
    it implies that $(M,\tilde{g})$ is Einstein and $\varphi_{\tilde{g}}=\varphi_++\varphi_-$, where $\varphi_\pm$ is a $\pm b$-Killing spinor for some $b\in\mathbb{R}\backslash\{0\}$.

Our next step is to show that one of $\varphi_+$ and $\varphi_-$ must vanish, and hence $\varphi_{\tilde g}$ is a Killing spinor. The proof is non-trivial, even on $\SS^n$, where Frank-Loss \cite{Frank_Loss_2024} spent more effort to prove it. We show this point slightly differently, but following crucially a key idea in \cite {FL1} that the vacuum is unique, which uses the structure of the Clifford algebra. We leave the proof in the next section.

    For the reverse,  by the confomal invaraince we may assume that $(M,{g})$ admits a non-parallel real Killing spinor $\varphi$ which solves \eqref{eq1.1}.  
   We need to show that it achieves equality. First of all, the existence of a Killing spinor implies $(M, g)$ is Einstein.
      Let the real Killing number of $\varphi$ be $0\neq b\in\mathbb{R}$.  Without loss of generality, we may assume $b=-\frac 12 $. It is well known that (see for instance \cite{Baum_Book_90})
    \eq{
    \D\varphi = \frac n 2 \,\varphi \quad \hbox{and} \quad {\rm R}\equiv n(n-1).
    }
       Combining it with \eqref{eq1.1} we have\eq{    \frac{n^2}{4}\abs{\varphi}^2 = \abs{\D\varphi}^2 = \lambda^2\abs{A}^2\abs{\varphi}^2.    }
    Thus $\lambda^2\abs{A}^2 = \frac{n^2}{4}$ and
    \eq{\label{eq_our_BHL}
    \lambda^n = \int \lambda^n \abs{A}^n = \Big(\frac{n}{2}\Big)^n\rVol(M,g).
    }
    Hence
    \eq{
    \lambda^2 = \frac{n^2}{4} \rVol (M,g) ^{\frac 2n} =\frac n {4(n-1)} {\rm R}  \rVol( M,g)^{\frac 2 n}.
    }
    If $(M,g)=(\S^n,g_{{\rm st}})$, then
    \eq{
    {\rm R}\rVol( M,g)^{\frac 2 n} = n(n-1)\omega_n^{\frac{2}{n}} = Y(\S^n),
    }
   and hence $\lambda^2 = \frac n{4(n-1)} Y(\SS^n)$.
    If $(M,g)\not=(\S^n,g_{{\rm st}})$, by Obata's theorem (see \cites{Obata62,Obata71}), $g$ is the unique metric with constant scalar curvature in $[g]$ up to scaling. Hence $g$ is the Yamabe minimizer. Therefore,
        \eq{
    \lambda^2 =\frac n {4(n-1)} {\rm R} \rVol( M,g)^{\frac 2 n} = \frac{n}{4(n-1)}Y(M,[g]).
    }
    \end{proof}

The main result holds true for $n=2$ as well. In this case, every oriented surface is spin. 

\begin{theorem}\label{2_dim}
Let $(\Sigma^2,g)$ be a closed oriented Riemann surface.
If $(\Sigma,g)$ admits a non-trivial solution $(\varphi,A)$ to \eqref{eq1.1}, where
$\varphi\in L^2$ and $A\in L^q$ with $q>2$,
then
\eq{
 \lambda^2 > 2\pi\chi(\Sigma).
}

\end{theorem}

\begin{proof}
    We only need to consider $\chi(\Sigma)>0$, hence $\Sigma$ is a topological sphere and $\chi(\Sigma)=2$. 

    Since $\varphi\in L^2$ and $A\in L^q$ with $q>2$, we have $A\cdot\varphi\in L^{\frac{2q}{q+2}}$. Using \eqref{eq1.1} and the apriori estimate \eqref{apriori_estimate} we can see $\varphi\in L^q$. Therefore Lemma \ref{L^r} implies that $\varphi\in L^r$ for any $r>2$. Using \eqref{apriori_estimate} again we have $\varphi\in W^{1,2}$.

    Note that Lemma \ref{conformal_invariance} remains true. In fact, for $\tilde{g}=e^{2h}g$ we have
    \eq{
    \D_{\tilde{g}}(e^{-\frac{1}{2}h}\tilde{\varphi}) = e^{-\frac{3}{2}h}\widetilde{\D_g\varphi}
    }
    and the others are the same. In 2-dimensional case we have the Gauss curvature
    \eq{
    {\rm K}_{\tilde{g}} = e^{-2h}({\rm K}_g - \Delta_g h).
    }
  
    Recall our normalization $\int \abs{A}_g^2 = 1$. By the Gauss-Bonnet theorem it follows
    \eq{    
    \int (4\pi\abs{A}_g^2 - {\rm K}_g ) = 0.
    }
    Hence there exists some $u_0$ solving (weakly)
    \eq{
    -\Delta_g u_0 + {\rm K}_g = 4\pi\abs{A}_g^2.
    }
    For $\tilde{g}=e^{2u_0}g$ we have
    \eq{\label{dim_2_eq1}
    \frac{1}{2}{\rm R}_{\tilde{g}} = {\rm K}_{\tilde{g}} = e^{-2u_0}({\rm K}_g-\Delta_g u_0) = e^{-2u_0}4\pi\abs{A}_g^2 = 4\pi\abs{A_{\tilde{g}}}_{\tilde{g}}^2,
    }
    where we have used the analogy of \eqref{A}.

    For simplicity of notation we omit the subscript $\tilde{g}$ for now. Identity \eqref{eq1.4} implies
    \eq{\label{ineq4.8}
    \int \abs{\nabla\varphi}^2 = \frac{1}{2}\int\abs{\D\varphi}^2 + \int\abs{P\varphi}^2 \geq \frac{1}{2}\int\abs{\D\varphi}^2.
    }
    On the other hand, the Schr\"odinger-Lichnerowicz formula \eqref{Bochner} implies
    \eq{
    \int \abs{\nabla\varphi}^2 = \int \<-\Delta\varphi,\varphi\> = \int\<\D^2\varphi - \frac{{\rm R}}{4}\varphi,\varphi\> = \int\abs{\D\varphi}^2 - 2\pi\int \abs{A}^2\abs{\varphi}^2.
    }
    Altogether we have
    \eq{
    4\pi\int \abs{A}^2\abs{\varphi}^2 \leq \int\abs{\D\varphi}^2 = \lambda^2\int \abs{A}^2\abs{\varphi}^2.
    }
    Hence $\lambda^2\geq 4\pi = 2\pi\chi(\Sigma)$.

    Assume equality holds.     Equality in \eqref{ineq4.8}
    implies that $\varphi$ is a twistor spinor, hence is smooth.
    Moreover, \eqref{why_varphi_constant_lenth} and \eqref{dim_2_eq1} imply that $\abs{\varphi}\equiv{\rm constant}\not=0$.
    For any orthonormal basis $\{e_j\}$ we have
    \eq{
    \<A,e_j\>\abs{\varphi}^2 = \rRe\<A\cdot\varphi,e_j\cdot\varphi\> = -\lambda^{-1}\rRe\<i\D\varphi,e_j\cdot\varphi\> = -\lambda^{-1}\rIm\<e_j\cdot\D\varphi,\varphi\>
      }
    Hence
    \eq{
    A = -\lambda^{-1}
    \abs{\varphi}^{-2}\rIm\<e_j\cdot\D\varphi,\varphi\> e_j.
    }
    Therefore $A$ is a nowhere vanishing smooth vector field, which is impossible on a topological sphere.

\end{proof}

\begin{remark}
    The case $\varphi\in L^2$ and $A\in L^2$  is a critical case. Recently,  Da Lio-Rivi\`ere proved in \cite{Riviere20} 
    that in this case $\varphi \in L^q$ for any $ q\ge 2$, if in addition ${\rm div}\, A=0$. Due to the gauge invariance of the zero mode equation, Lemma \ref{lem2.1}, we have this additional condition for free, after changing a gauge. Hence Theorem \ref{2_dim} holds also true if $A\in L^2$.
\end{remark}

We remark that on $\SS^2$ the existence of zero modes is well understood by the result of Aharobnov and Casher  \cite{Aharonov-Casher19} in terms of $ \curl A$. See also \cite{Erdoes-Solovej01}.

\section{Zero modes and Sasaki-Einstein manifolds}
\label{zero_modes_and_Sasaki-Einstein_manifolds}

In this section, we discuss the relation between 
Sasaki-Einstein manifolds and ``optimal" zero modes, i.e., those achieve equality in
\eqref{ineq0}, and 
complete the proof of the main theorem.

We recall that for an odd-dimensional manifold, an \textit{almost contact structure} $(g,\xi,\eta,\Phi)$, consisting of a Riemannian metric $g$, a vector field $\xi$, a $1$-form $\eta$ and a $(1,1)$-tensor field $\Phi$, is defined by
\eq{
\eta(\xi)=1,\quad \Phi^2(X) = -X+\eta(X)\xi,\quad g(\Phi(X),\Phi(Y)) = g(X,Y)-\eta(X)\eta(Y), \quad\forall X,Y\in\Gamma(TM).
}
Moreover, it is a \textit{Sasakian structure} if in addition
\eq{
(\nabla_X\Phi)(Y) = g(X,Y)\xi - \eta(Y)X, \quad\forall X,Y\in\Gamma(TM).
}
In this case, $\xi$ is called its \textit{Reeb field}.

We call an odd-dimensional Riemannian manifold \textit{Sasaki-Einstein} if it is Einstein and has a Sasakian structure.
On a Sasaki-Einstein manifold, $\xi$ is a canonical vector field. We will prove that $\xi$ plays the role as the vector field $A$ in the zero mode equation \eqref{eq1.1}.

\begin{theorem}\label{rest_of_proof}
    Let $(M^n,g)$ be an $n$-dimensional ($n \ge 2$) closed Riemannian spin manifold and $\varphi$ a non-trivial solution of \eqref{eq1.1} for some vector field $A$. If $\varphi$, in addition, is a Killing spinor, then we have 
    \begin{enumerate}
        \item $A$ is a Killing vector field with constant length and $\nabla_A A=0$;
        \item $n$ is odd;
        \item $(M,g)$ is Sasaki-Einstein with $\frac {2\lambda}n A$ as its Reeb field.
    \end{enumerate}
\end{theorem}

{The proof is similar to \cite{Friedrich-Kath89_JDG}*{Theorem 2}, where they used a slightly different definition of Sasakian structure.}

\begin{proof}

    (1)
    Since $\varphi$ is a Killing spinor, without loss of generality we may assume  that $\varphi$ is a $-\frac{1}{2}$-Killing spinor with $\abs{\varphi}\equiv1$.  Since it solves \eqref{eq1.1},  we have
    \eq{\label{Killing_zero_mode}
    \frac n 2 \varphi = \D\varphi =i\lambda A\cdot \varphi.
    }
    Therefore
    \eq{\label{constants}
    \abs{A} \equiv \frac{n}{2\lambda}.
    }
    For any orthonormal basis $\{e_j\}$ we have
    \eq{
    \<A,e_j\> = \rRe\<A\cdot\varphi,e_j\cdot\varphi\> = -\lambda^{-1}\rRe\<i\D\varphi,e_j\cdot\varphi\> = -\lambda^{-1}\rIm\<e_j\cdot\D\varphi,\varphi\> = -\frac{n}{2\lambda}\rIm\<e_j\cdot\varphi,\varphi\>.
    }
    Hence
    \eq{\label{local_A}
    A = -\frac{n}{2\lambda}\rIm\<e_j\cdot\varphi,\varphi\> e_j.
    }
    Since $\varphi$, as a Killing spinor, is smooth, we see that $A$ is also smooth.
    
    At each point $p\in M$ we may assume $\nabla e_j|_p=0$. It is easy to see     \eq{
    \<\nabla_{e_k}A,e_j\> &= -\frac{n}{2\lambda}\rIm\<e_j\cdot\nabla_{e_k}\varphi,\varphi\> -\frac{n}{2\lambda}\rIm\<e_j\cdot\varphi,\nabla_{e_k}\varphi\>\\
    &= \frac{n}{4\lambda}\rIm\<e_j\cdot e_k\cdot\varphi,\varphi\> + \frac{n}{4\lambda}\rIm\<e_j\cdot\varphi,e_k\cdot\varphi\>\\
    &= \frac{n}{4\lambda}\rIm\<e_j\cdot e_k\cdot\varphi,\varphi\> - \frac{n}{4\lambda}\rIm\<e_k\cdot e_j\cdot\varphi,\varphi\>\\
    &= -\<\nabla_{e_j}A,e_k\>.
    }
    Therefore $A$ is a Killing vector field with $\abs{A} \equiv \frac{n}{2\lambda}$. It follows easily that $\nabla_A A=0$. 
    
    (2) 
    Set
    \eq{
    \xi \coloneqq \frac {2\lambda}n A.
    }
    From (1) we see that
    \eq{
    \abs{\xi}=1, \quad \nabla_{\xi}\xi=0, \quad \varphi = i\xi\cdot\varphi.
    }
    For any vector field $X\in\Gamma(TM)$, differentiating the last identity gives
    \eq{
    \nabla_X\varphi = i(\nabla_X \xi)\cdot\varphi + i\xi\cdot\nabla_X\varphi.
    }
    Since $\varphi$ is a $-\frac{1}{2}$-Killing spinor, we have
    \eq{
    -\frac{1}{2}X\cdot\varphi &= i(\nabla_X \xi)\cdot\varphi - \frac{i}{2}\xi\cdot X\cdot\varphi\\
    &= i(\nabla_X \xi)\cdot\varphi + i\<\xi,X\>\varphi + \frac{i}{2}X\cdot \xi\cdot\varphi\\
    &= i(\nabla_X \xi)\cdot\varphi + i\<\xi,X\>\varphi + \frac{1}{2} X\cdot\varphi.
    }
    Hence
    \eq{\label{eq_xi}
    (\nabla_X \xi)\cdot\varphi = -\<\xi,X\>\varphi + iX\cdot\varphi.
    }
   We define an endomorphism $\Phi:\Gamma(TM)\to \Gamma(TM)$ by
    \eq{
    \Phi(X) \coloneqq -\nabla_X \xi.
    }
    Denote by $\xi ^\perp$ the orthogonal complement  of the line bundle generated by $\xi$, i.e.
    \eq{
    \xi ^\perp \coloneqq \bigsqcup_{p\in M} \xi (p)^\perp,
    }
    where
    \eq{
    \xi (p)^\perp \coloneqq \{ X\in T_p M \,:\, \<\xi(p) ,X\>=0 \} \subset T_p M.
    }
    It is clear that for $X\perp\xi $, \eqref{eq_xi} implies
    \eq{\label{Killing_zero_mode_2}
    (\nabla_X \xi )\cdot\varphi = i X\cdot\varphi.\quad\forall X\in\Gamma(\xi ^\perp).
    }
    Moreover since $|\xi|=1$, we have $\<\Phi (X), \xi \>=0$. Hence $\Phi$ induces  an endomorphism $\Phi:\Gamma(\xi ^\perp)\to \Gamma(\xi ^\perp)$ by restriction. 
It is clear that \eqref{Killing_zero_mode_2} is equivalent to
    \eq{\label{Killing_zero_mode_3}
    \Phi(X)\cdot\varphi = -iX\cdot\varphi.
    }
    It follows
    \eq{\Phi^2(X)\cdot\varphi = -i\Phi(X)\cdot \varphi=-X\cdot \varphi}
    and hence
    \eq{
    \left(\Phi^2(X)+X\right)\cdot\varphi = 0.
    }
    Since $\varphi$ has constant length, we obtain
    \eq{
    \Phi^2 = -\rid \quad\hbox{on}\quad \Gamma(\xi ^\perp).
    }
    It is possible only when ${\rm rank}(\xi ^\perp) = n-1$ is even, namely $n$ is odd.

    (3) 
  Now assume $n$ is odd. We define 
     the dual one-form to $\xi$ by
    \eq{
    \eta(X)\coloneqq \<X, \xi\>,     \quad\forall X\in\Gamma(TM).
    }
       From (2) we know that
    \eq{
    \Phi(\xi) = -\nabla_\xi \xi   
    = 0.
    }
    Now we check that $(M,g,\xi,\eta,\Phi)$ is Sasaki-Einstein. It is well-known that the existence of a Killing spinor implies the manifold is Einstein. Hence it suffices to check that it is Sasakian.
    In the following, we always assume $X,Y\in\Gamma(\xi ^\perp)$.
    
    First, it is clear 
    \eq{
    \eta(\xi) =  1, \qquad 
    \eta(X) = \<X,\xi\> = 0.
    }
    Second,
    \eq{
    \Phi^2(X) = -X = -X + \eta(X)\xi, \qquad 
    \Phi^2(\xi) = 0 = -\xi + \eta(\xi)\xi.
    }
    Third,
    \begin{align}
    \<\Phi(X),\Phi(Y)\> &= \<\nabla_X \xi ,\nabla_Y \xi \> = \rRe\<\nabla_X \xi \cdot\varphi, \nabla_Y \xi \cdot\varphi\>\\
    &= \rRe\<i X\cdot\varphi, i Y\cdot\varphi\> = \<X,Y\> = \<X,Y\> - \eta(X)\eta(Y),\\    
    \<\Phi(X),\Phi(\xi)\> &= 0 = \<X,\xi\> - \eta(X)\eta(\xi),\\
    \<\Phi(\xi),\Phi(\xi)\> &= 0 = \<\xi,\xi\> - \eta(\xi)\eta(\xi).
    \end{align}
    Hence $(g,\xi,\eta,\Phi)$ is an almost contact structure. 
    
    Moreover, for any $X\in\Gamma(\xi ^\perp)$ and $V\in\Gamma(TM)$, we decompose
    \eq{
    \nabla_VX = (\nabla_VX)^T + f\xi \quad\hbox{for some function}\ f, 
    }
    where $(\nabla_VX)^T\in\Gamma(\xi ^\perp)$. 
    Differentiating \eqref{Killing_zero_mode_3} gives
    \begin{align}
        -i\nabla_V X\cdot\varphi + \frac{i}{2}X\cdot V\cdot\varphi &= \nabla_V\Phi(X)\cdot\varphi - \frac{1}{2}\Phi(X)\cdot V\cdot\varphi\\
        &= (\nabla_V\Phi)(X)\cdot\varphi + \Phi(\nabla_VX)\cdot\varphi + \frac{1}{2}V\cdot \Phi(X)\cdot\varphi + \<\Phi(X),V\>\varphi\\
        &= (\nabla_V\Phi)(X)\cdot\varphi - i(\nabla_VX)^T\cdot\varphi - \frac{i}{2}V\cdot X\cdot\varphi + \<\Phi(X),V\>\varphi.
    \end{align}
    Hence
    \eq{
    (\nabla_V\Phi)(X)\cdot\varphi &= -if\xi\cdot\varphi - i\<X,V\>\varphi - \<\Phi(X),V\>\varphi\\
    &= -f\varphi - i\<X,V\>\varphi - \<\Phi(X),V\>\varphi.
    }
    Given any orthonormal basis $\{e_j\}$, we have
    \eq{
    \<(\nabla_V\Phi)(X),e_j\> = \rRe\<(\nabla_V\Phi)(X)\cdot\varphi,e_j\cdot\varphi\> = -\<X,V\>\rRe\<i\varphi,e_j\cdot\varphi\> = -\<X,V\>\rIm\<e_j\cdot\varphi,\varphi\>.
    }
    Hence
    \eq{
    (\nabla_V\Phi)(X) = -\<X,V\>\rIm\<e_j\cdot\varphi,\varphi\>e_j = \<X,V\>\xi = \<X,V\>\xi - \eta(X)V,
    }
    where we have used \eqref{local_A}.
    Finally
    \eq{
    (\nabla_V\Phi)(\xi) = -\Phi(\nabla_V\xi) = \Phi^2(V) = -V + \eta(V)\xi = \<\xi,V\>\xi - \eta(\xi)V.
    }
    Therefore
    \eq{
    (\nabla_V\Phi)(W) = \<W,V\>\xi - \eta(W)V, \quad\forall V,W\in\Gamma(TM)
    }
    and hence we complete the proof.

\end{proof}

\begin{lemma}\label{lem7.11}
    Let $\varphi=\varphi_++\varphi_-$, where $\varphi_\pm$ is a $\mp  \frac  12 $-Killing spinor.
    If $(\varphi,A)$ is a zero mode with $\abs{\varphi}$ constant, then $\<\varphi_+,\varphi_-\>=0$ and
    $\varphi_+$ and $\varphi_-$ are both zero modes with the same $A$ and $\lambda$, i.e.
    \eq{\D \varphi _\pm =i\lambda A\cdot \varphi_\pm.
    }
    In particular,
    \eq{ \label{decomp} \frac n 2 \varphi_+=   i\lambda A\cdot \varphi_+, \quad 
     -\frac n 2 \varphi_-=   i\lambda A\cdot \varphi_-.
    }
\end{lemma}

\begin{proof}
       Each Killing spinor has constant length, so do $\varphi_+$ and $\varphi_-$.
    Since $\abs{\varphi}\equiv{\rm const}$, we have $\rRe\<\varphi_+,\varphi_-\>\equiv{\rm const}$. But $\varphi_+$ is $L^2$-orthogonal to $\varphi_-$, hence
    \eq{\label{lem_eq1}
    \rRe\<\varphi_+,\varphi_-\>=0.
    }
    Since $\varphi_\pm$ is a $\mp \frac 12 $-Killing spinor, we have
    \eq{
    \D\varphi_+ = \frac n 2 \varphi_+, \quad \D\varphi_- = -\frac n2 \varphi_-.
    }
    Hence the zero mode equation \eqref{eq1.1} implies
    \eq{\label{eq4.4a}
    \frac n 2 \varphi_+  -\frac n 2 \varphi_- = i\lambda A\cdot\varphi_+ + i\lambda A\cdot\varphi_-.
    }
      From \eqref{lem_eq1} and \eqref{eq4.4a}, it is elementary to see
       \eq{\label{hermitian_1}
    \frac {n^2} 4 \abs{\varphi_+}^2 &= \rRe\<  \frac n 2 \varphi_+ -\frac n 2 \varphi_- , \frac n2 \varphi_+\>
    = \rRe\< i\lambda A\cdot\varphi_+ + i\lambda A\cdot\varphi_-, \frac n2 \varphi_+ \>\\
    &= \rRe\< i\lambda A\cdot\varphi_+ , \frac n 2 \varphi_+ \> + \rRe\< i\lambda A\cdot\varphi_-, \frac n 2 \varphi_+ \>
    }
    and
        \eq{\label{hermitian_3}
    \lambda^2\abs{A}^2\abs{\varphi_+}^2 &= \rRe\<i\lambda A\cdot\varphi_+, i\lambda A\cdot\varphi_+ + i\lambda A\cdot\varphi_-\> = \rRe\<i\lambda A\cdot\varphi_+, \frac n2  \varphi_+ -\frac n 2 \varphi_-\>\\
    &= \rRe\<i\lambda A\cdot\varphi_+, \frac n  2 \varphi_+\> - \rRe\< i\lambda A\cdot\varphi_-, \frac n2\varphi_+ \>.
    }
       It follows that 
    \eq{
    \abs{\frac n 2 \varphi_+ - i\lambda A\cdot\varphi_+}^2
        = \frac{ n^2} 4  \abs{\varphi_+}^2 + \lambda^2\abs{A}^2\abs{\varphi_+}^2 - 2\rRe\<i\lambda A\cdot\varphi_+, \frac n2 \varphi_+\>  = 0.
    }
    Hence $ \D\varphi_+ =\frac n2 \varphi_+ = i\lambda A\cdot\varphi_+$ and 
    $\D\varphi_- = -\frac n2\varphi_- = i\lambda A\cdot\varphi_-$ by  \eqref{eq4.4a}.

    Finally, we show that $\<\varphi_+,\varphi_-\>=0$. Indeed, from \eqref{decomp}we have
    \eq{
    \frac  n 2 \,\rIm\<\varphi_+,\varphi_-\> = \rIm\<i\lambda A\cdot\varphi_+,\varphi_-\>
    }
    and 
    \eq{
    \frac n 2 \,\rIm\<\varphi_+,\varphi_-\> = -\rIm\<\varphi_+,i\lambda A\cdot\varphi_-\>
    = -\rIm\<i\lambda A\cdot\varphi_+,\varphi_-\>,
    }
    where in the last equality we have used the fact that $i\lambda A\cdot$ is symmetric.
    Hence $\rIm\<\varphi_+,\varphi_-\> = 0$.
    Together with \eqref{lem_eq1} we complete the proof.
    
\end{proof}

The following Proposition, together with Lemma \ref{lem7.11} and Theorem \ref{rest_of_proof}, completes the proof of Theorem \ref{thm7.12}.

\begin{proposition}\label{uniqueness}
    Let $(M,g,\xi,\eta,\Phi)$ be a Sasaki-Einstein manifold. If it admits a $\mp\frac{1}{2}$-Killing spinor $\varphi_{\pm}$ such that $\<\varphi_+,\varphi_-\>=0$ and 
    \eq{\label{uniqueness_eq1}
    \varphi_+ = i\xi\cdot\varphi_+,\quad \varphi_- = -i\xi\cdot\varphi_-,
    }
    then either $\varphi_+=0$ or $\varphi_-=0$.
\end{proposition}

\begin{proof}
    As in the proof of \eqref{eq_xi}, differentiating \eqref{uniqueness_eq1} yields
    \eq{\label{uniqueness_eq2}
    (\Phi(X)+iX)\cdot\varphi_+ = 0,\quad\forall X\perp\xi
    }
    and similarly
    \eq{\label{uniqueness_eq3}
    (\Phi(X)+iX)\cdot\varphi_- = 0,\quad\forall X\perp\xi.
    }

      Hence $\varphi_+$ and $\varphi_-$ are both vacuums in the sense of Definition 
    \ref{def_vacuum} below.
    Without loss of generality, we may assume $\varphi_+$ is not zero.
    The uniqueness of the vacuum, Lemma \ref{lem_vacuum} below, then implies that
    \[ \varphi_- =f\varphi_+,\]
    for a complex function $f$. Now the assumption implies that $0 = \<\varphi_+,\varphi_-\> = f\abs{\varphi_+}^2 $. Hence $f=0$ and $\varphi_-=0$. This completes the proof.

\end{proof}

Inspired by the definition of vacuum in \cite{FL22}, we introduce
\begin{definition}\label{def_vacuum}
    Let $(M,g,\xi,\eta,\Phi)$ be a Sasakian spin manifold. A nowhere vanishing spinor field $\varphi$ is called a \textit{vacuum} w.r.t. $\Phi$, if
    \eq{
    (\Phi(X)+iX)\cdot\varphi = 0, \quad\hbox{for any vector field}\ X\perp\xi.
    }
\end{definition}

\begin{lemma}\label{lem_vacuum}
Let $\varphi$ be a vacuum.  
Then it is unique in the sense that if $\phi$ is another vacuum, then
\[ \phi = f\varphi\]
for some complex function $f$.

\end{lemma} 
\begin{proof}
This is in fact a local property and then follows from the proof given in \cite{FL1}. For convenience of the reader we sketch  the proof.

Fix any point $p\in M$. Without loss of generality we may normalize $\abs{\varphi(p)}=1$. 
We choose an adapted orthonormal basis
    $\{e_j\}_{j=1}^{n}$ around $p$ such that 
    \eq{
    e_n \coloneqq \xi(p), \quad e_{2j} \coloneqq \Phi(e_{2j-1}), \quad 1\leq j\leq m\coloneqq\frac{n-1}{2},
    }
         and introduce the Clifford actions $F_j,F_j^*:\Sigma_p M\to\Sigma_p M$ defined by
    \eq{\label{uniqueness_eq4}
    F_j \coloneqq \frac{1}{2}(e_{2j-1} + ie_{2j})\cdot\,,
    \quad 
    F^*_j \coloneqq \frac{1}{2}(e_{2j-1} - ie_{2j})\cdot.
    }
    It is elementary to check that
    \eq{
    F_j F_k + F_k F_j = 0  =F_j^* F_k^* +F_k^* F_j^* , \quad F_j F_k^*  +F_k^*   F_j = -\delta_{jk}, \quad \forall\, 1\leq j,k\leq m,
    }
    and
    \eq{
    \<F_j\psi_1,\psi_2\> + \<\psi_1,F_j^*\psi_2\> = 0 \quad\hbox{for any two spinors }\psi_1,\psi_2.
    }
    Since $\varphi$ is a vacuum, we have
    \eq{\label{v1}
    F_j^*\varphi(p) = 0,\quad \forall\, 1\leq j \leq m.
    }
    As in \cite{FL1}*{Appendix A}, one can now show that the spinors $F_{j_1}\cdots F_{j_l}  \varphi(p)$ ($0\le l\le m$) form an orthogonal basis of $\Sigma_pM$. 
   If there is another vacuum $\phi$, not parallel to $\varphi$, then we may assume that  $\phi (p)\perp \varphi (p) $.
    It follows easily  that $\varphi$ is unique in the above sense. However, from \eqref{v1} we have
    \eq{\langle \phi(p), F_{j_1} \cdots F_{j_l} \varphi\rangle =\langle F_{j_1}^*  \phi (p) , F_{j_2}\cdots F_{j_l} \varphi \rangle =0,} 
   which is a contradiction. \end{proof}

\begin{remark}\label{rmk4.6}
    (1) Equivalently,
    \eq{\label{basis_rmk4.6}
    \{ e_{j_1}\cdot\cdots\cdot e_{j_l}\cdot\psi \,|\, 0\leq l \leq m, 1\leq j_1 < \cdots < j_l \leq 2m-1, \hbox{each } j_k \hbox{ is odd}\}
    }
    is an orthonormal basis of $\Sigma_pM$.

    (2) Note that if $(g,\xi,\eta,\Phi)$ is a Sasakian structure, then $(g,-\xi,-\eta,-\Phi)$ is also a Sasakian structure. Therefore, $\varphi$ defined by
    \eq{
    (-\Phi(X)+iX)\cdot\varphi = 0, \quad\hbox{for any vector field}\ X\perp\xi,
    }
    is the vacuum w.r.t. $-\Phi$, hence also unique.
\end{remark}

  Now we prove our second main result, which verifies the second ``if and only if" in Theorem \ref{main_thm}.

\begin{theorem}\label{thm5.4}
    Let $(M^n,g)$ $(n\geq2)$ be a closed spin manifold.     Then
    \eq{\label{ineq_thm5.2}
        \lambda^2 \geq \frac{n}{4(n-1)} Y(M,[g]).
    }
    Equality holds if and only if $(M,g)$ is Sasaki-Einstein with $A$ (up to scaling) as its Reeb field and $\varphi$  is a vacuum up to a conformal transformation.
 \end{theorem}

\begin{proof}
     The ``only if'' part directly follows from Theorem \ref{thm7.12}, Theorem \ref{2_dim}, Theorem \ref{rest_of_proof} and Proposition \ref{uniqueness}. 

     Now we prove the ``if'' part.
     Assume (up to a conformal transformation) $(M,g)$ is a Sasaki-Einstein manifold with $A$ (or after renormalization $\xi=\frac{2\lambda}{n}A$) as its Reeb field and $\varphi$ is a vacuum. 

     Let $(\tilde{M},\tilde{g})$ be the universal cover of $(M,g)$. Then $(\tilde{M},\tilde{g})$ is simply connected and also Sasaki-Einstein. We lift $(\varphi,A, \xi, g)$ to $\tilde M$ and denote them by $(\tilde \varphi,  \tilde A, \tilde \xi, \tilde g)$.
     It is clear that $(\tilde{\varphi},\tilde{A})$ is also a zero mode and $\tilde{\varphi}$ is also a vacuum on $(\tilde{M},\tilde{g})$.
A result from Friedrich-Kim \cite{Friedrich_Kim_2000}*{Theorem 6.3 and Corollary 6.1} implies that $(\tilde{M},\tilde{g})$, up to scaling, admits a non-parallel $-\frac{1}{2}$-Killing spinor $\tilde \psi$ such that (for short we still denote $\tilde{\cdot}$ by $\cdot$)
    \eq{ \label{eq_4.9a}
        \tilde \psi = i\tilde \xi \cdot \tilde \psi.
    }
  
     Differentiating it as in the proof of Theorem \ref{rest_of_proof}, we see that $\tilde \psi$ is a vacuum.
   The uniqueness,  Lemma \ref{lem_vacuum},  then implies that
    \eq{\label{up_to_a_complex_function}
    \tilde{\varphi} = f \tilde \psi
    }
    for some complex function $f$ on $\tilde M$. 
         Using the zero mode equation \eqref{eq1.1}, i.e., $\tilde\D  \tilde \varphi =i\tilde A\cdot \tilde \varphi$, one can see that \eqref{up_to_a_complex_function} yields 
       \eq{
    i\lambda \tilde A\cdot\tilde{\varphi} = \tilde\D\tilde{\varphi} = f\tilde\D\tilde\psi + \rd f\cdot\tilde\psi = \frac{n}{2}f\tilde\psi + \rd f\cdot\tilde\psi = i\lambda f \tilde A\cdot \tilde\psi + \rd f\cdot\tilde\psi = i\lambda \tilde A\cdot\tilde{\varphi} + \rd f\cdot\tilde\psi
    }
    where we have used \eqref{eq_4.9a} in the fourth equality. 
    It follows $\rd f\cdot \tilde\psi = 0$. Therefore $\rd f=0$ since $\tilde\psi$ has non-zero constant length, and hence $f={\rm const}$. 
    Therefore, $\tilde{\varphi}$ (hence $\varphi$) is a non-parallel real Killing spinor. The conclusion then follows from Theorem \ref{thm7.12}.
\end{proof}

    \begin{remark} On the odd-dimensional sphere $\SS^n$, our proof recovers the solutions given by Loss-Yau \cite{Loss_Yau_86} and Dunne-Min \cite{Dunne_Min_08}. The slight difference to the proof given in \cite{Frank_Loss_2024} lies in the decomposition of the spinor field $\varphi$ by $\varphi_+$ and $\varphi_-$, which are $\mp \frac 12 $ Killing spinors w.r.t  the spherical metric $\frac 4 {(1+|x|^2 ) ^2} dx^2$. Then the uniqueness of vacuum shows that either $\varphi_+=0$ or $\varphi_-=0$.
\end{remark}

Finally, Theorem \ref{thm7.12}, Theorem \ref{2_dim} and Theorem \ref{thm5.4} complete the proof of our main results, Theorem \ref{main_thm} and Theorem \ref{main_thm_2_dim}.

\section{Generalizations}\label{sec5}

Our approach can be generalized to consider the following problem. 
Let $k=0, 1, \cdots, n$ and $\alpha\in\Omega^k$ be a $k$-form. We consider the following  $k$-th ``zero mode'' equation
\eq{\label{general_zero_mode}
\D \varphi = i^{[\frac{k+1}{2}]} \alpha \cdot \varphi,
}
where the choice of power $[\frac{k+1}{2}]$ ensures that $i^{[\frac{k+1}{2}]}\alpha\cdot$ is symmetric, namely
\eq{
\<i^{[\frac{k+1}{2}]}\alpha\cdot\psi_1,\psi_2\> = \<\psi_1,i^{[\frac{k+1}{2}]}\alpha\cdot\psi_2\>.
}
Here we use the definition of a $k$-form acting on a spinor field as follows (see also \cite{Baum_Book_90})
\eq{
\alpha\cdot \varphi
\coloneqq \sum_{1 \leq i_1 < \cdots < i_k \leq n}
  \alpha(e_{i_1}, \cdots, e_{i_k})
  \, e_{i_1} \cdot \cdots \cdot e_{i_k}\cdot \varphi, \quad \forall\,\alpha\in\Omega^k.
}
If $\alpha=f$ is a $0$-form, it is clear that $\alpha \cdot \varphi = f\varphi$.
Since we are interested in finding a lower bound of $\int \abs{\alpha}^n$, without loss of generality, we may assume that $
 \int \abs{\alpha}^n<\infty.$
 Similarly we  consider the following equivalent problem
 \eq{\label{k-form_eq}
 \D \varphi = i^{[\frac{k+1}{2}]}\lambda \alpha\cdot \varphi}
 for a $k$-form $\alpha\in L^n$ with $\int \abs{\alpha}^n =1$, $\lambda>0$ and $\varphi\not\equiv0$.

\begin{lemma}[Conformal invariance]\label{k-form_conformal_invariance}
     Equation \eqref{k-form_eq} is conformally invariant in the following sense: Let $(\varphi,\alpha)$ be a solution to \eqref{k-form_eq} w.r.t. metric $g$ and
     $\tilde g =h^{\frac 4{n-2}} g$. 
     Define 
     \eq{
     \varphi_{\tilde{g}} \coloneqq h^{-\frac{n-1}{n-2}}\tilde{\varphi}, \qquad \alpha_{\tilde{g}} \coloneqq  h^{-\frac{2}{n-2}}\tilde{\alpha} := h^{\frac{2k-2}{n-2}}\alpha.
     } 
     Then $(\varphi_{\tilde g},\alpha_{\tilde g})$ is a solution to \eqref{k-form_eq} w.r.t. metric $\tilde{g}$, i.e.
     \eq{
\D_{\tilde{g}}\varphi_{\tilde{g}} = i^{[\frac{k+1}{2}]}\lambda  \alpha_{\tilde{g}}\,\tilde\cdot\,\varphi_{\tilde{g}},
     }
     with the same $\lambda$.
     \end{lemma} 
     
     \begin{proof}
   Similarly as in Section \ref{sec3}  we have \eq{
    \D_{\tilde{g}}\varphi_{\tilde{g}} &= h^{-\frac{n+1}{n-2}}\widetilde{\D_g\varphi} = i^{[\frac{k+1}{2}]}\lambda h^{-\frac{n+1}{n-2}}\widetilde{\alpha\cdot\varphi} = i^{[\frac{k+1}{2}]}\lambda h^{-\frac{n+1}{n-2}}\tilde{\alpha}\,\tilde{\cdot}\,\tilde{\varphi}\\
    &= i^{[\frac{k+1}{2}]}\lambda h^{-\frac{n+1}{n-2}}\cdot h^{\frac{2}{n-2}}\alpha_{\tilde{g}}\,\tilde{\cdot}\,h^{\frac{n-1}{n-2}}\varphi_{\tilde{g}} = i^{[\frac{k+1}{2}]}\lambda \alpha_{\tilde{g}}\,\tilde{\cdot}\,\varphi_{\tilde{g}}.
    }
    Moreover, we have
     \eq{  \abs{\alpha_{\tilde{g}}}_{\tilde g}  = h^{-\frac{2k}{n-2}} \abs{h^{\frac{2k-2}{n-2}}\alpha}_g   = h^{\frac {-2}{n-2}}\abs{\alpha}_g , 
    } and hence
    \eq{
    \int \abs{\alpha_{\tilde{g}}}^n_{\tilde g}  \rd V_{\tilde{g}} = \int \big( h^{\frac{-2}{n-2}} \abs{\alpha}_g \big)^n \cdot h^{\frac{2n}{n-2}}\, \rd V_g = \int \abs{\alpha}_g^n \rd V_g = 1.
    } 
    \end{proof}

With this Lemma we can define a family of conformal invariants w.r.t. $\alpha$ as in Section \ref{sec3}, which are in fact the same and which  is related to the ordinary Yamabe constant $Y(M,[g])$.
\begin{remark}
    In particular, for $k=0$, $\alpha=f$ is a function and $f_{\tilde{g}}=h^{-\frac{2}{n-2}}f$. For $k=1$, $\alpha$ is a 1-form and $\alpha_{\tilde{g}}=\alpha$. Let  $ \alpha ^\sharp $  be the dual vector of $\alpha$ w.r.t.  $g$ and
     $ \alpha ^{\sharp_{\tilde g}} $  be the dual vector of $\alpha$ w.r.t. $\tilde g$. We have
    \eq{
    \alpha_{\tilde{g}}^{\sharp_{\tilde{g}}}=h^{-\frac{4}{n-2}}\alpha_{\tilde{g}}^\sharp=h^{-\frac{4}{n-2}}\alpha^\sharp,
    }
    which coincides with the conformal change for vector fields given in  Lemma \ref{conformal_invariance}.
\end{remark}


The case $k=1$ was studied before. Now we consider the case $k=0$ and $k=2$ separately. These cases are in principle quite similar to case $k=1$ considered above, but are slightly different.
Case $k\ge 3$ involves the subtle algebraic structure  of spinors and differential forms, therefore  we leave it  in a forthcoming paper.

\subsection{Case \texorpdfstring{$k=0$}{k=0}} This case has its own  interest. In this case we 
consider
\eq{\label{eq_k=0}
\D\varphi = f\varphi,
}
where $f\in L^n$ is a function.

Our goal is to prove
\eq{\label{casek=0}
\norm{f}_n^2 \geq \frac{n}{4(n-1)}Y(M,[g]).
}
It is also convenient to consider the following equivalent problem
 \eq{\label{0-form_eq}
 \D \varphi = \lambda f\varphi}
 for a function $f\in L^n$ with $\int f^n =1$, $\lambda>0$ and $\varphi\not\equiv0$. The desired inequality can be written as
 \eq{\label{0-form_ineq}
 \lambda^2 \ge\frac{n}{4(n-1)} Y(M, [g]).
 }

 \begin{theorem}\label{thm5.3}
     Let $(M^n,g)$ $(n\geq2)$ be a closed manifold with a spin structure $\sigma$ and
$(\varphi,f)$ a non-trivial solution to \eqref{0-form_eq} with $\varphi\in L^p$  $(p>\frac{n}{n-1})$.
Then
\eq{
\lambda^2 \ge \frac {n}{4(n-1)} Y(M,[g]),
}
with equality if and only if $f$ is constant and $\varphi$ is a non-parallel real Killing spinor up to a conformal transformation.
 \end{theorem}
We remark that in the equality case, $(M,g)$  needs not to be Sasakian.  In fact, on any $n$ sphere $\SS^n$, there is $f$ (in fact $f$ is a constant) such that 
\eqref{eq_k=0} has a non-trivial solution and equality in \eqref{casek=0} holds. See \cite{Frank_Loss_2024}.

 \begin{proof}[Proof of Theorem \ref{thm5.3}]
  We use the same argument in Section \ref{sec3} by replacing $\abs{A}^2$ by $f^2 $ and insert \eqref{0-form_eq}
  into \eqref{eq_b1} to obtain
  \eq{
0= \int ({{\rm R}} - \frac{4(n-1)}{n}\lambda^2f^2)\abs{\varphi}^2 +4 \int \abs{P\varphi}^2,
} like \eqref{reminder_twistor}. Then we follow the contradiction argument given in Section \ref{sec3} to
  obtain inequality \eqref{0-form_ineq}, as  in Remark \ref{rmk}. We skip it, 
    and only prove the equality case. Assume equality holds. As in the proof of Theorem \ref{thm7.12}, we have
     \eq{
     {\rm R}_{\tilde g} = Y(M,[g])f_{\tilde g}^2 = \frac{4(n-1)}{n}\lambda^2 f_{\tilde g}^2, \qquad P_{\tilde g}\varphi_{\tilde g}=0
     }
     for some $\tilde g\in[g]$. Therefore $\varphi_{\tilde g}$ is a twistor spinor and hence by \eqref{why_varphi_constant_lenth} one can see $\abs{\varphi_{\tilde g}}_{\tilde g}\equiv{\rm const}$ and $(M,\tilde g)$ is Einstein.
     For simplicity of notation we omit the subscript $\tilde g$.
     Since $(M,g)$ is Einstein, ${\rm R}\equiv{\rm const}$, hence $f\equiv{\rm const}$. Using Obata's theorem we have
     \eq{
     {\rm R}\equiv Y(M,[g])\rVol^{-\frac{2}{n}},\qquad f\equiv \pm\rVol^{-\frac{1}{n}}.
     }
     It follows
     \eq{
     \D\varphi = \lambda f\varphi = \pm\sqrt{\frac{n{\rm R}}{4(n-1)}}\varphi.
     }
     Using the definition of twistor spinor, one can see that $\varphi$ is a $\mp\sqrt{\frac{{\rm R}}{4n(n-1)}}$-Killing spinor.
     
If $\varphi$ is a $b$-Killing spinor with $b\not=0$, then
     \eq{
     \D\varphi = -nb\varphi = \mp\sqrt{\frac{n{\rm R}}{4(n-1)}}\varphi = \mp\sqrt{\frac{n}{4(n-1)}Y(M,[g])}\rVol^{-\frac{1}{n}}\varphi \eqcolon \lambda f\varphi.
     }
     Hence we have equality and  complete the proof.
 \end{proof}

 \begin{remark}
     When $(M,g)=(\SS^n,g_{{\rm st}})$, Theorem \ref{thm5.3} was proved in \cite{Frank_Loss_2024}. We remark that due to the conformal invariance, the problem in $\SS^n$ is equivalent to the one in $\R^n$. If $f=1$, \eqref{0-form_eq} is the eigenvalue equation for the Dirac operator in $(M,g)$ and Theorem \ref{thm5.3} is in fact the result of Hijazi \cite{H86}.
     If $f=|\varphi|^{\frac 2{n-1}}$, it was proved in \cite{Jurgen_Julio-Batalla}, with a similar approach.
 \end{remark}

\subsection{Case \texorpdfstring{$k=2$}{k=2}} Now in this subsection we consider $k=2$, i.e., 
\eq{
\D\varphi = i\alpha\cdot\varphi,
}
where $\alpha$ is a 2-form. 
Remember  a $k$-form acts on a spinor field as follows
\eq{
\omega \cdot \varphi
\coloneqq \sum_{1 \leq i_1 < \cdots < i_k \leq n}
  \omega(e_{i_1}, \cdots, e_{i_k})
  \, e_{i_1} \cdot \cdots \cdot e_{i_k}\cdot \varphi, \quad \forall\,\omega\in\Omega^k.
}
In particular, we have
\eq{
X^\flat\cdot\varphi = X\cdot\varphi, \quad\forall X\in\Gamma(TM)
}
and
\eq{
X^\flat\w Y^\flat\cdot\varphi = X\cdot Y\cdot\varphi, \quad\forall X\perp Y.
}

Our goal is to prove
\eq{
\norm{\alpha}_n^2 \geq \Big[\frac{n}{2}\Big]^{-1} \frac{n}{4(n-1)}Y(M,[g]).
}
Since we are interested in finding a lower bound of $\int \abs{\alpha}^n$, without loss of generality, we may assume as above that $
 \int \abs{\alpha}^n<\infty.$
 Similarly we  consider the following equivalent problem
 \eq{\label{2-form_eq}
 \D \varphi = i\lambda \alpha\cdot \varphi}
 for a 2-form $\alpha\in L^n$ with $\int \abs{\alpha}^n =1$, $\lambda>0$ and $\varphi\not\equiv0$. Then the desired inequality  can be written as
 \eq{\label{2-form_ineq}
 \lambda^2 \ge\frac{n}{4m(n-1)} Y(M, [g]),
 }
 where $m\coloneqq [\frac{n}{2}]$.

We will need the following  elementary lemma. The proof is elementary.
\begin{lemma}\label{schur_lem}
    For any anti-symmetric linear operator $A:\R^n\to\R^n$, there exists an orthonormal basis such that w.r.t. the basis
    \eq{
    A = \operatorname{diag}\!\left[
    \begin{pmatrix} 0 & -s_1 \\ s_1 & 0 \end{pmatrix}, \dots,
    \begin{pmatrix} 0 & -s_m \\ s_m & 0 \end{pmatrix}, 0
    \right], \quad\hbox{if $n$ is odd}
    }
    and
    \eq{
    A = \operatorname{diag}\!\left[
    \begin{pmatrix} 0 & -s_1 \\ s_1 & 0 \end{pmatrix}, \dots,
    \begin{pmatrix} 0 & -s_m \\ s_m & 0 \end{pmatrix}
    \right], \quad\hbox{if $n$ is even},
    }
    where $s_j\geq0$ for $1\leq j\leq m$ and $m = [\frac{n}{2}]$.
\end{lemma}

The following elementary Lemma is the key for our result. 
\begin{lemma}\label{sharp_length_constant} 
    On the Euclidean space $\R^n$,
    for any 2-form $\alpha$ and any spinor field $\varphi$, we have
    \eq{\label{lem5.7_ineq1}
    \abs{\alpha\cdot\varphi}^2 \leq m\abs{\alpha}^2\abs{\varphi}^2,
    }
    with equality if and only if 
    \eq{\label{lem5.7_eq1}
    \alpha = c\sum_{j=1}^m e_{2j-1}^\flat\w e_{2j}^\flat
   }
    for some constant $c\in\R$ and
    \eq{\label{lem5.7_eq2}
    e_{2j-1}\cdot e_{2j}\cdot\varphi = \phi, \quad\forall \,1\leq j\leq m
    }
    for some spinor field $\phi$.
    Here the basis is chosen as in the previous Lemma, if one views $\alpha$ as an anti-symmetric linear operator.
\end{lemma}
\begin{proof}
    In view of Lemma \ref{schur_lem}, without loss of generality we may assume 
    \eq{
    \alpha = \sum_{j=1}^m c_j e_{2j-1}^\flat\w e_{2j}^\flat \quad\hbox{with}\quad c_j\in\R,
    } and hence $|\alpha|^2 =\sum_{j=1}^m c_j^2$.
    It follows
    \eq{
    \abs{\alpha\cdot\varphi}^2 &=  \abs{\sum_{j=1}^m c_j e_{2j-1}\cdot e_{2j} \cdot\varphi}^2 \leq \Big(\sum_{j=1}^m \abs{c_j e_{2j-1}\cdot e_{2j} \cdot\varphi}\Big)^2 \\& = \Big(\sum_{j=1}^m c_j\abs{\varphi}\Big)^2 \leq m\sum_{j=1}^m c_j^2\abs{\varphi}^2 = m\abs{\alpha}^2\abs{\varphi}^2,
    }
    where we have used the triangle inequality in the first inequality and the Cauchy-Schwarz inequality in the second.
    It is easy to see that equality holds if and only if $c_1=\cdots=c_m$ and
    \eq{
    e_{2j-1}\cdot e_{2j} \cdot\varphi = \phi, \quad\forall j,
    }
    for some $\phi$. 
    
\end{proof}

We denote by $\Psi(\alpha)$ the corresponding $(1,1)$-tensor of $\alpha$, i.e.
\eq{
\Psi(\alpha)(X)\coloneqq X\ip\alpha, \quad\forall X\in\Gamma(TM).
}
Now we prove
\begin{theorem}
Let $(M^n,g)$ $(n\geq3)$ be a closed manifold with a spin structure $\sigma$ and let
$(\varphi,\alpha)$ be a non-trivial solution to \eqref{2-form_eq} with $\varphi\in L^p$  $(p>\frac{n}{n-1})$.
Then
\eq{\label{2-form-ineq}
\lambda^2 \ge \frac {n}{4m(n-1)} Y(M,[g]), \quad m \coloneqq \Big[\frac{n}{2}\Big],
}
with equality 
if and only if $(M,g,\xi,\eta, \Phi)$ with  $\xi=\lambda n^{-1}(\rd^*\alpha)^\sharp$,  $\eta=\lambda n^{-1}\rd^*\alpha$ and $ \Phi=-2m\lambda n^{-1}\Psi(\alpha)$ is Sasaki-Einstein
and $\varphi$ is
\begin{enumerate}
    \item[(a)] a Killing spinor and vacuum, if $n\equiv3\!\!\mod4$;
    \item[(b)] $\varphi=\varphi_++\varphi_-$ with $\varphi_\pm$ a vacuum (w.r.t. $\pm\Phi$) and $\mp\frac{1}{2}$-Killing spinor if $n\equiv1\!\!\mod4$;
\end{enumerate}
up to a conformal transformation.
\end{theorem}

\begin{proof}
 Again we use the same argument in Section \ref{sec3} by 
 replacing $\abs{A}^2$ by $m|\alpha|^2 $ and insert \eqref{2-form_eq}
  into \eqref{eq_b1} to obtain
  \eq{
0= \int ({{\rm R}}|\varphi|^2  - \frac{4(n-1)}{n}\lambda^2\abs{\alpha \cdot \varphi}^2) +4 \int \abs{P\varphi}^2,
} like \eqref{reminder_twistor}. 
Now here there exists one small difference. We need to use estimate \eqref{lem5.7_ineq1} to get
 \eq{
0 \ge  \int ({{\rm R}}|\varphi|^2  - \frac{4(n-1)}{n}\lambda^2 m\abs{\alpha}^ 2  \abs{ \varphi}^2) +4 \int \abs{P\varphi}^2,
}
Now we can follow the contradiction argument as in Section \ref{sec3} to
  obtain inequality \eqref{2-form-ineq}, replacing $|A|^2$ by $m|\alpha|^2$. We leave the proof to the interested reader.

Assume equality holds.
An analogy of Proposition \ref{cor7.9} implies that up to a conformal transformation $\abs{\alpha}\equiv{\rm const}$.
The same argument as in the proof of Theorem \ref{thm7.12} implies $\varphi$ is a twistor spinor with constant length, hence $\varphi=\varphi_++\varphi_-$ with $\varphi_\pm$ a $\pm b$-Killing spinor. After scaling we may assume $b=-\frac{1}{2}$. Since at least one of $\varphi_\pm$ is not zero, without loss of generality we may assume $\abs{\varphi_+}=1$.

By Lemma \ref{sharp_length_constant} we have $e_{2j-1}\cdot e_{2j}\cdot \varphi = \phi$, then $\alpha\cdot\varphi = mc\phi$, hence
\eq{
    \frac{n}{2}(\varphi_+ - \varphi_-) = \D\varphi = i\lambda\alpha\cdot\varphi 
        = i\lambda mc\phi.
}
It follows
\eq{
    e_{2j-1}\cdot e_{2j}\cdot (\varphi_+ + \varphi_-) = \phi = -i\frac{n}{2\lambda mc}(\varphi_+ - \varphi_-).
}
Now argue as in Lemma \ref{lem7.11} we obtain
\eq{
    e_{2j-1}\cdot e_{2j}\cdot\varphi_+ = -i\frac{n}{2\lambda mc}\varphi_+, \qquad e_{2j-1}\cdot e_{2j}\cdot\varphi_- = i\frac{n}{2\lambda mc}\varphi_-
}
and hence
\eq{\label{eq0_thm5.9}
    \alpha\cdot\varphi_+ = -i\frac{n}{2\lambda}\varphi_+, \qquad \alpha\cdot\varphi_- = i\frac{n}{2\lambda}\varphi_-.
}
Recall that for a 2-form $\alpha$ we have (c.f. \cite{Baum_Book_90})
\eq{\label{eq2_thm5.9}
X\cdot\alpha\cdot = (X^\flat\w\alpha - X\ip\alpha)\cdot, \quad
\alpha\cdot X\cdot = (X^\flat\w\alpha + X\ip\alpha)\cdot, \quad\forall X\in\Gamma(TM).
}
It follows
\eq{\label{eq22_thm5.9}
(X^\flat\w\alpha)\cdot = \frac{1}{2}(\alpha\cdot X\cdot + X\cdot\alpha\cdot), \quad (X\ip\alpha)\cdot = \frac{1}{2}(\alpha\cdot X\cdot - X\cdot\alpha\cdot).
}
Using this fact, differentiating \eqref{eq0_thm5.9} yields
\eq{\label{eq1_thm5.9}
\nabla_X\alpha\cdot\varphi_+ = X\ip\alpha\cdot\varphi_+,
\quad \nabla_X\alpha\cdot\varphi_- = -X\ip\alpha\cdot\varphi_-,
\quad\forall X\in\Gamma(TM).
}
Define
\eq{
\xi\coloneqq \lambda n^{-1}\rd^*\alpha.
}
Now we prove that \eq{\label{vaccum2}\xi\cdot\varphi_\pm = -i\varphi_\pm.}
In fact, \eqref{eq22_thm5.9} implies
\eq{\label{eq3_thm5.9}
\rd^*\alpha\cdot\varphi_+ = -\sum_{j=1}^n(e_j\ip\nabla_{e_j}\alpha)\cdot\varphi_+ 
= \frac{1}{2}\sum_{j=1}^n(e_j\cdot\nabla_{e_j}\alpha\cdot\varphi_+ - \nabla_{e_j}\alpha\cdot e_j\cdot\varphi_+)
}
and
\eq{\label{eq4_thm5.9}
\rd\alpha\cdot\varphi_+ = \sum_{j=1}^n(e_j^\flat\w\nabla_{e_j}\alpha)\cdot\varphi_+ 
= \frac{1}{2}\sum_{j=1}^n(e_j\cdot\nabla_{e_j}\alpha\cdot\varphi_+ + \nabla_{e_j}\alpha\cdot e_j\cdot\varphi_+).
}
Using \eqref{eq0_thm5.9}, \eqref{eq1_thm5.9} and \eqref{eq2_thm5.9} we have
\eq{\label{eq5_thm5.9}
\sum_{j=1}^n e_j\cdot\nabla_{e_j}\alpha\cdot\varphi_+ = \sum_{j=1}^n e_j\cdot(e_j\ip\alpha)\cdot\varphi_+ = \sum_{j=1}^n (e_j^\flat\w e_j\ip\alpha)\cdot\varphi_+ = 2\alpha\cdot\varphi_+ = -i\lambda^{-1}n\varphi_+.
}
We claim $\rd \alpha=0$.
We start the proof of  the claim by defining a 2-form $\beta$ by
\eq{\label{def_beta}
\beta = c\sum_{j_1<j_2} \rIm\<e_{j_1}\cdot e_{j_2}\cdot\varphi_+,\varphi_+\>e_{j_1}^\flat\w e_{j_2}^\flat.
}
It is clear that $\beta$ is well-defined. Since $\varphi_+$ is a $-\frac{1}{2}$-Killing spinor, differentiating \eqref{def_beta} yields
\eq{
\nabla_{e_j}\beta  = \frac{c}{2}\Big(-\rIm\<e_{j_1}\cdot e_{j_2}\cdot e_j\cdot\varphi_+, \varphi_+\>-\rIm\<e_{j_1}\cdot e_{j_2}\cdot \varphi_+, e_j\cdot\varphi_+\>\Big)e_{j_1}^\flat\w e_{j_2}^\flat.
}
Using $e_j\cdot e_k\cdot + e_k\cdot e_j\cdot = -2\delta_{jk}$, it is elementary to check that
\eq{
\nabla_{e_j}\beta = 2c\,\rIm\<e_{j_1}\cdot\varphi_+,\varphi_+\>e_{j_1}^\flat\w e_j^\flat,
}
and hence $d\beta=0$. Now we show that $\alpha=\beta$.
\eqref{eq0_thm5.9} and Lemma \ref{sharp_length_constant} imply that at any $p\in M$, for an adapted basis $\{e_j\}$ we have
\eq{
e_{2j-1}\cdot e_{2j}\cdot (\varphi_+ + \varphi_-) = -ic_0(\varphi_+ - \varphi_-), \quad\forall \,1\leq j\leq m
}
for some constant $c_0$. Comparing the length we see that $c_0=\pm1$. By changing the sign of $c$ if necessary, we may assume $c_0=-1$. Arguing as in Lemma \ref{lem7.11} we have
\eq{\label{eq_two_vacuum_thm5.9}
    e_{2j-1}\cdot e_{2j}\cdot\varphi_+ = i\varphi_+, \qquad
    e_{2j-1}\cdot e_{2j}\cdot\varphi_- = -i\varphi_-, \quad\forall \,1\leq j\leq m.
} 
Combining with \eqref{lem5.7_eq1} we have
\eq{
c = -\frac{n}{2m\lambda}.
}
The same argument as in Lemma \ref{lem_vacuum} implies that \eqref{basis_rmk4.6} is an orthonormal basis of $\Sigma_pM$.
Hence
    \eq{
    \rIm\<e_{j_1}\cdot e_{j_2}\cdot\varphi_+,\varphi_+\> &= 1 \quad\hbox{if}\quad (j_1,j_2)=(2j-1,2j) \quad\hbox{for some}\quad 1\leq j \leq m,\\ 
    \rIm\<e_{j_1}\cdot e_{j_2}\cdot\varphi_+,\varphi_+\> &= 0 \quad\hbox{otherwise}.
    }
Therefore \eq{\beta=  c \sum_{j=1}^m \rIm\<e_{2j-1} \cdot e_{2j }\cdot\varphi_+,\varphi_+\>e_{2j-1}^\flat\w e_{2j}^\flat = c \sum_{j=1}^m e_{2j-1}^\flat\w e_{2j}^\flat=\alpha.}  
Therefore $\rd\alpha=0$, the claim.
Now combining with \eqref{eq3_thm5.9}, \eqref{eq4_thm5.9} and \eqref{eq5_thm5.9} we have
\eq{
\rd^*\alpha\cdot\varphi_+ = -i\lambda^{-1}n\varphi_+.
}
Hence by the  definition of $\xi$ we have $\xi\cdot\varphi_+ = -i\varphi_+$. Similarly we have $\xi\cdot\varphi_- = -i\varphi_-$.
Moreover, by \eqref{lem5.7_eq1} and \eqref{eq_two_vacuum_thm5.9} we have
\eq{
(c^{-1}\Psi(\alpha)(e_j)+ie_j)\cdot\varphi_+ = 0, \quad\forall 1\leq j\leq 2m.
}
Now Theorem \ref{rest_of_proof} implies that $(M,g,\xi,\eta, \Phi)$ with  $\xi=\lambda n^{-1}(\rd^*\alpha)^\sharp$,  $\eta=\lambda n^{-1}\rd^*\alpha,$ and $ \Phi=-2m\lambda n^{-1}\Psi(\alpha)$ is Sasaki-Einstein. 

It remains to show that either $\varphi_+=0$ or $\varphi_-=0$ when $n\equiv3\!\!\mod4$. In this case, $m$ is odd.
We denote the volume form by
\eq{
\omega \coloneqq e_1^\flat\w\cdots\w e_{2m}^\flat\w e_{2m+1}^\flat.
}
Since $\omega^2=\rid$, 
\eq{
\omega\cdot = \epsilon\,\rid,
}
where $\epsilon\in\{1,-1\}$ is determined by the choice of the  Clifford representation.
Therefore
\eq{
e_{2m+1}\cdot\varphi_+ = \epsilon (-i)^m\varphi_+, \qquad e_{2m+1}\cdot\varphi_- = \epsilon\, i^m\varphi_-.
}
We define a vector field $\bar{\xi}$ by
\eq{\label{def_canonical_vector}
\bar{\xi} \coloneqq \epsilon (-1)^{\frac{m+1}{2}} \sum_{j=1}^n \rIm\<e_j\cdot\varphi_+,\varphi_+\>e_j,
}
which is independent of the choice of basis, hence well defined.
In view of \eqref{eq_two_vacuum_thm5.9} we see that $\bar{\xi}=e_{2m+1}$ at any point $p\in M$ and hence
\eq{
\bar{\xi}\cdot\varphi_+ = -i\varphi_+, \quad \bar{\xi}\cdot\varphi_- = i\varphi_-.
}
Recall that
\eq{
\xi\cdot\varphi_+ = -i\varphi_+ = \bar{\xi}\cdot\varphi_+, \quad \xi\cdot\varphi_- = -i\varphi_- = -\bar{\xi}\cdot\varphi_-.
}
Since $|\varphi_+|=1$, we have $\xi=\bar\xi$, and hence $\xi\cdot\varphi_-=0$ and  $\varphi_-=0$.

The other direction is easy to show.
Again we may assume the Killing number is $b=-\frac{1}{2}$.
In both cases we have 
\eq{
2\alpha = \sum_{j=1}^n e_j\w e_j\ip\alpha = \sum_{j=1}^n e_j\w\Psi(\alpha)(e_j) =  -\frac{n}{2m\lambda}\sum_{j=1}^n e_j\w\Phi(e_j).
}
Therefore $\alpha$ has the same form as \eqref{lem5.7_eq1} and
\eq{
\abs{\alpha}^2 = m\Big(-\frac{n}{2m\lambda}\Big)^2 = \frac{n^2}{4m\lambda^2}.
}
Since the normalization $\norm{\alpha}_n=1$, we can show as the last paragraph  of the proof for  Theorem \ref{thm7.12}
\eq{
\lambda^2 = \frac{n^2}{4m}\rVol^{\frac{2}{n}} = \frac{n}{4m(n-1)}Y(M,[g]),
}
in view of the fact that $g$ is a Yamabe minimizer.
\end{proof}

\begin{remark}

In the proof, \eqref{vaccum2} implies that $\varphi_+$ and $\varphi_-$ are vacuums, but one w.r.t. $\Phi$, another w.r.t. $-\Phi$. Hence one can not use the uniqueness of vacuum to show that one of them must be zero. In fact,  as stated in the Theorem, in the case $n\equiv 1$ mod $4$, both can be non-zero, and $\varphi$ itself is not necessarily a Killing spinor.

\end{remark}

 \begin{remark}\label{rmk5.11} We give three remarks for the equality case.

    (1) The Reeb field $\xi$ can also be defined by    
    \eq{
    \xi \coloneqq c[*(\alpha\w\alpha\w\cdots\w\alpha)]^\sharp \quad\hbox{(m times)}.
    }

    (2) Using the same argument as in Theorem \ref{thm5.4} and \cite{Friedrich_Kim_2000}*{Theorem 6.3 and Corollary 6.1}, we can in fact prove that: if $(M,g,\xi,\eta,\Phi)$ is Sasaki-Einstein and $\varphi$ is
    \begin{enumerate}
    \item[(a)] a vacuum if $n\equiv3\!\!\mod4$;
    \item[(b)] the sum of a vacuum w.r.t. $\Phi$ and a vacuum w.r.t. $-\Phi$ if $n\equiv1\!\!\mod4$;
\end{enumerate}
then $(\varphi,\alpha)$ is a zero mode achieving equality, where $\alpha$ is determined by $\Phi$ as above.

(3) According to \cite{Friedrich_Kim_2000}*{Theorem 6.3 and Corollary 6.1} or C. B\"ar \cite{Bar1993}, when $n\equiv3\!\!\mod4$, we have in fact $\varphi_-=0$ and hence $\varphi=\varphi_+$ is a $-\frac{1}{2}$-Killing spinor up to a conformal transformation.
\end{remark}

 \subsection{Case \texorpdfstring{$k=0$}{k=0} plus \texorpdfstring{$k=1$}{k=1}}
 Just before posting the paper into ArXiv, we have learned from Jonah Reu\ss{} a paper
by V. Branding, N. Ginoux and G. Habib \cite{BGH25}.
They considered (in our notation) the following equation
\eq{\label{BGH}\D\varphi=iA\cdot \varphi+\lambda \varphi, \qquad \varphi \not= 0}
for a one form $A$ and a constant $\lambda$ and obtained
\eq{
(|\lambda| \rVol(M,g)^{\frac{1}{n}} + \|A\|_{L^n})^2 \ge \frac  n{4(n-1)}Y(M,[g]),
}
by using a similar method presented in \cite{Frank_Loss_2024} and \cite{Reuss25}, which uses the $\epsilon$-regularization $\sqrt{|\varphi|^2 +\epsilon^2}$. If equality holds, they showed that  up to a conformal transformation $(M,g)$ is Sasaki-Einstein (in the case $A\not =0$) and $\varphi$ is a twistor with constant length and hence $\varphi=\varphi_++\varphi_-$ with $\pm b$-Killing spinor $\varphi_\pm$. Their case  is also included in  our generalization.

 We  can do slightly generally by considering 
 \eq{\label{eq_k=0_plus_k=1}
 \D\varphi = iA\cdot\varphi + f\varphi,
 }
 where $A$ is  a one-form (or a vector field) and $f$ is a function, i.e., combining case $k=0$ and  case $k=1$.  We denote 
 \eq{
 \alpha \coloneqq iA + f.
 } 
 We may assume $A$
 and $f$  both are not zero, otherwise it was considered above.

 Note that by Lemma \ref{conformal_invariance} and Lemma \ref{k-form_conformal_invariance} we have
 \eq{
 \abs{\alpha_{\tilde{g}}}_{\tilde{g}}^2 = \abs{A_{\tilde{g}}}_{\tilde{g}}^2 + f_{\tilde{g}}^2 = h^{-\frac{4}{n-2}}\abs{A}_g^2 + h^{-\frac{4}{n-2}}f^2 = h^{-\frac{4}{n-2}}\abs{\alpha}^2.
 }
It is clear that
 \eq{\label{eq1_k=0_plus_k=1}
 \abs{\alpha\cdot\varphi}^2&=\abs{iA\cdot\varphi + f\varphi}^2 = (\abs{A}^2+f^2)\abs{\varphi}^2+2\rRe\<iA\cdot\varphi,f\varphi\>\\
 &\leq (\abs{A}^2+f^2)\abs{\varphi}^2+ 2\abs{A}\abs{\varphi}\cdot\abs{f}\abs{\varphi} \leq (\abs{A}+\abs{f})^2\abs{\varphi}^2.
 } 
 Therefore, using the same argument as in Section \ref{sec3}  and replacing $|A|^2$ by $(|A|+|f|)^2$
 we immediately have 
 \eq{\label{ineq_k=0_plus_k=1}
 \norm{\abs{A}+\abs{f}}_{L^n}^2 \geq \frac{n}{4(n-1)}Y(M,[g]),
 }
 which implies the case considered in \cite{BGH25}.

 Now assume that equality holds, then again up to a conformal transformation $\varphi$ is a twistor spinor with constant length and $\abs{A}+\abs{f}\equiv{\rm const}\not =0$.

 Assume now $A$ is nowhere vanishing. Then 
  equality in \eqref{eq1_k=0_plus_k=1} implies
 \eq{ f\varphi= h
 \,iA\cdot\varphi
 }
 for some non-negative function $h$. It follows
 $\abs{f} = h\abs{A}$
 and hence
 $(1+h)\abs{A}=|A|+|f|\equiv{\rm const}\not =0$. 
 Together with \eqref{eq_k=0_plus_k=1} we have
 \eq{
 \D\varphi = i(1+h)A\cdot\varphi
 }
 and
 \eq{
 \norm{(1+h)A}_{L^n}^2 = \frac{n}{4(n-1)}Y(M,[g]).
 }
 It  goes back to Theorem \ref{main_thm} with $\tilde A:=(1+h)A$.
Hence equality in \eqref{ineq_k=0_plus_k=1} holds if and only if $\varphi$ is a non-parallel real Killing spinor up to a conformal transformation, and if and only if 
$(M,g)$ is a Sasaki-Einstein manifold with $(1+\abs{f}\abs{A}^{-1})A$ as its Reeb field and $\varphi$  a vacuum up to a conformal transformation.

Assume $f$ is nowhere vanishing. Then 
  equality in \eqref{eq1_k=0_plus_k=1} implies
 \eq{
 iA\cdot\varphi = h f\varphi
 }
 for some non-negative function $h$. It follows similarly that 
 $\abs{A}=h \abs{f} $
 and hence
 $(1+h)\abs{f}=|A|+|f|\equiv{\rm const}\not =0$. 
 Together with \eqref{eq_k=0_plus_k=1} we have
 \eq{
 \D\varphi = (1+h) f \varphi
 }
 and
 \eq{
 \norm{(1+h)f}_{L^n}^2 = \frac{n}{4(n-1)}Y(M,[g]).
 }
 It  goes back to case $k=0$ with $\tilde f:=(1+h)f$. If we further assume $f$ is a constant, then $h$ is also a non-zero constant, and hence $A$ is nowhere vanishing. It goes back to the case $k=1$ again.

\subsection{Further generalization}

Our approach can be generalized to deal with a general endomorphism. We consider
\eq{\label{endomorphism_eq}
    \D\varphi = \lambda H\varphi,
}
where $H\in \Gamma({\rm End}^{{\rm sym}}(\Sigma M))$ is symmetric and normalized as $\int_M \norm{H}^n = 1$. Here we denote by $\norm{H}$ the operator norm of $H$ and by $\norm{H}_n$ the $L^n$-norm of $\norm{H}$.  

In this general case, we have the similar conformal invariance. In fact, for the isometry $F: (\Sigma M,g) \to (\Sigma M, \tilde{g})$ with $\varphi \mapsto \tilde{\varphi}$ given in Section \ref{sec2.2}, we define the induced $\tilde{H}$ by
\eq{
    \tilde{H} \coloneqq F \circ H \circ F^{-1} : \Gamma(\Sigma M,\tilde{g})\to \Gamma(\Sigma M,\tilde{g}).
}
It is easy to check that Lemma \ref{conformal_invariance} holds true for
\eq{\label{conformal_H}
     \varphi_{\tilde{g}} \coloneqq h^{-\frac{n-1}{n-2}}\tilde{\varphi}, \qquad H_{\tilde{g}} \coloneqq h^{-\frac{2}{n-2}}\tilde{H}
}
and for $n=2$ with $\tilde{g}=e^{2h}$.
Note that by the definition of operator norm we have
\eq{\label{def_operator_norm}
    \norm{H} = \sup_{\varphi\not\equiv0} \frac{\abs{H\varphi}}{\abs{\varphi}}.
}
Hence for any spinor field $\varphi$ we have
\eq{
    \abs{H\varphi} \leq \norm{H}\cdot\abs{\varphi}.
}
In view of Remark \ref{rmk}, using the same argument as in Section \ref{sec3} we can prove the following theorem.

\begin{theorem}\label{generalization_thm}
Let $(M^n,g)$ $(n\geq2)$ be a closed spin manifold. If $(M,g)$ admits a non-trivial solution $(\varphi,H)$ to \eqref{endomorphism_eq}, where $\varphi\in L^p$ with $p>\frac{n}{n-1}$ and $\norm{H}\in L^n$ with $\norm{H}_n = 1$, then
\eq{\label{ineq_thm5.10}
    \lambda^2 \geq \frac {n}{4(n-1)} Y(M,[g]).
}
Equality holds if and only if (up to a conformal transformation) $\varphi = \varphi_+ + \varphi_-$, where $\varphi_\pm$ is a $\mp \frac{1}{2}$-Killing spinor with $\rRe\<\varphi_+,\varphi_-\>=0$, and $\varphi$ achieves the supreme in \eqref{def_operator_norm}.

\end{theorem}

\begin{proof}
We use the same argument in Section \ref{sec3} by replacing $f^2$ by $\norm{H}^2$ and using the integral Schr\"odinger-Lichnerowicz formula \eqref{Sch-Lich} to obtain
\eq{
    0 \geq \int ({\rm R} - \frac{4(n-1)}{n}\lambda^2\norm{H}^2)\abs{\varphi}^2 + 4\int \abs{P\varphi}^2.
}
Then we follow the contradiction argument given in Section \ref{sec3} to obtain inequality \eqref{ineq_thm5.10}, as  in Remark \ref{rmk}. We skip it, and only prove the equality case.

Assume equality holds, then $\varphi$ achieves the supreme in \eqref{def_operator_norm}, i.e. $\abs{H\varphi}=\norm{H}\cdot\abs{\varphi}$.
As in the proof of Theorem \ref{main_thm}, we have
\eq{
    {\rm R}_{\tilde g} = Y(M,[g])\norm{H_{\tilde{g}}}_{\tilde g}^2 = \frac{4(n-1)}{n}\lambda^2 \norm{H_{\tilde{g}}}_{\tilde g}^2, \qquad P_{\tilde g}\varphi_{\tilde g}=0
}
for some $\tilde g\in[g]$. Therefore $\varphi_{\tilde g}$ is a twistor spinor and hence by \eqref{why_varphi_constant_lenth} one can see $\abs{\varphi_{\tilde g}}_{\tilde g}\equiv{\rm const}$ and $(M,\tilde g)$ is Einstein. By a well known result (see for instance \cite{Ginoux09}*{Proposition A.2.1}) it implies that $(M,\tilde{g})$ is Einstein and $\varphi_{\tilde{g}}=\varphi_++\varphi_-$, where $\varphi_\pm$ is a $\pm b$-Killing spinor for some $b\in\mathbb{R}\backslash\{0\}$.

For simplicity we omit the subscript $\tilde{g}$ and without loss of generalization, we assume $b=-\frac{1}{2}$. Thus ${\rm R}\equiv n(n-1)$, and again by Obata's Theorem we have
\eq{
    Y(M, [g]) = {\rm R}\cdot\rVol^{\frac{2}{n}} = n(n-1)\rVol^{\frac{2}{n}}.
}
Therefore the equality in \eqref{ineq_thm5.10} implies
\eq{\label{eq_thm5.2_value_of_lambda}
    \lambda^2 = \frac{n}{4(n-1)}Y(M, [g]) = \frac{n^2}{4}\rVol^{\frac{2}{n}}.
}
Since $\abs{\varphi}\equiv{\rm const}$, we have $\rRe\<\varphi_+,\varphi_-\>\equiv{\rm const}$. Now the equation \eqref{endomorphism_eq} is
\eq{\label{decomposition_endomorphism_eq}
    \frac{n}{2}(\varphi_+ - \varphi_-) = \D\varphi = \lambda H\varphi = \lambda H(\varphi_+ + \varphi_-).
}
Taking length yields
\eq{
    \frac{n^2}{4}(\abs{\varphi_+}^2 + \abs{\varphi_-}^2 - 2\rRe\<\varphi_+,\varphi_-\>) = \lambda^2 \norm{H}^2 (\abs{\varphi_+}^2 + \abs{\varphi_-}^2 + 2\rRe\<\varphi_+,\varphi_-\>).
}
Recall that $\norm{H}_n=1$, taking the square of $L^n$-norm of \eqref{decomposition_endomorphism_eq} yields
\eq{
    \frac{n^2}{4}\rVol^{\frac{2}{n}}(\abs{\varphi_+}^2 + \abs{\varphi_-}^2 - 2\rRe\<\varphi_+,\varphi_-\>) = \lambda^2 (\abs{\varphi_+}^2 + \abs{\varphi_-}^2 + 2\rRe\<\varphi_+,\varphi_-\>).
}
Together with \eqref{eq_thm5.2_value_of_lambda} we have
\eq{
    \rRe\<\varphi_+,\varphi_-\> = 0.
}

For the reverse, due to the conformal invariance of \eqref{endomorphism_eq}, we may assume $\varphi = \varphi_+ + \varphi_-$, where $\varphi_\pm$ is a $\mp \frac{1}{2}$-Killing spinor with $\rRe\<\varphi_+,\varphi_-\>=0$, and $\varphi$ achieves the supreme in \eqref{def_operator_norm}. It follows
\eq{
    \lambda H\varphi = \D\varphi = \D(\varphi_+ + \varphi_-) = \frac{n}{2}(\varphi_+ - \varphi_-).
}
Since $\rRe\<\varphi_+,\varphi_-\>=0$, we have
\eq{
    \abs{\varphi} = \abs{\varphi_+ + \varphi_-} = \abs{\varphi_+ - \varphi_-}
}
and hence
\eq{
    \lambda \norm{H}\cdot\abs{\varphi} = \lambda \abs{H\varphi} = \frac{n}{2}\abs{\varphi_+ - \varphi_-} = \frac{n}{2}\abs{\varphi}.
}
Therefore $\lambda\norm{H}=\frac{n}{2}$ and hence
\eq{
    \lambda^2 = \norm{\lambda\norm{H}}_n^2 = \norm{\frac{n}{2}}_n^2 = \frac{n^2}{4}\rVol^{\frac{2}{n}}.
}
Again the existence of a $\mp\frac{1}{2}$-Killing spinor implies ${\rm R}\equiv n(n-1)$, and
\eq{
    Y(M, [g]) = {\rm R}\cdot\rVol^{\frac{2}{n}} = n(n-1)\rVol^{\frac{2}{n}}.
}
Therefore
\eq{
    \lambda^2 = \frac{n^2}{4}\rVol^{\frac{2}{n}} = \frac{n}{4(n-1)}Y(M, [g]).
}

\end{proof}

\textbf{Example 1:} If $H = f$, where $f$ is a function, then
\eq{
    \norm{H} = \abs{f}.
}
It goes back to Case $k=0$.

\textbf{Example 2:} If $H = iA\cdot$, where $A$ is a vector field and $\cdot$ is the Clifford multiplication, then
\eq{
    \norm{H} = \abs{A}.
}
It goes back to Theorem \ref{main_thm} and Theorem \ref{main_thm_2_dim}.

\textbf{Example 3:} If $H = f + iA\cdot$, where $f$ is a function, $A$ is a vector field and $\cdot$ is the Clifford multiplication, then
\eq{
    \norm{H} \leq \abs{f} + \abs{A}.
}
It goes back to Case $k=0$ plus $k=1$.

\textbf{Example 4:} If $H = i\alpha\cdot$, where $\alpha$ is a differential $2$-form, then
\eq{
    \norm{H} = \big[\frac{n}{2}\big]^{\frac{1}{2}}\cdot\abs{\alpha}.
}
It goes back to Case $k=2$.

\section{A new conformal invariant}

 For any $(M,g)$  we define a new invariant
 \eq{\label{def_Y_v}
 Y_v(M,[g]) \coloneqq \inf \big\{\lambda^2 \, \big|\, \D_g \varphi =i\lambda A\cdot \varphi \ \hbox{for some} \ \varphi \not \equiv 0 \ \hbox{and vector field}\ A \ \hbox{with}\ \norm{A}_n=1 \big\}.
 }
 
 The discussion in Section \ref{Section_the_zero_mode_equation} shows that it is a conformal invariant.
 \begin{lemma}
     $Y_v$ is conformally invariant.
 \end{lemma}

 \noindent{\it Question}: {\it 
 Is it true that $Y_v(M,[g])\le Y_v(\mathbb{S}^n)$? }

 \

We conjecture that it is true at least when $n$ is odd.

The result of Frank-loss \cite{Frank_Loss_2024} implies that on the odd-dimensional standard sphere we have
\[
Y_v (\SS^n) = \frac{n}{4(n-1)} Y(\SS^n).
\]
Theorem \ref{main_thm} implies clearly
\eq{
    Y_v(M,[g]) \geq \frac{n}{4(n-1)} Y(M,[g])
}
on any $(M,g)$. 
If $(M,g)$ is a simply connected Sasakian-Einstein, then using Friedrich-Kim \cite{Friedrich_Kim_2000} and Theorem \ref{main_thm} we have
\eq{
Y_v(M,[g]) = \frac{n}{4(n-1)} Y(M,[g]).
}

It natural   to ask when  
equality holds and if  $Y_v $ is achieved.

\appendix

\section{The regularity} \label{sec_a}

\begin{lemma}\label{L^r}
    If $\varphi\in L^p$ is a solution to \eqref{eq1.1} with $p>\frac{n}{n-1}$, then $\varphi\in L^r$ for any $\frac{n}{n-1}<r<\infty$.
\end{lemma}

\begin{proof} The proof follows completely the one in \cite{FL1} with one exception that we use the ellipticity theory instead of the Hardy-Littlewood-Sobolev (HLS) inequality. For convenience of the reader we repeat it here.

First we fix $r$ with $\frac{n}{n-1}<r\leq\frac{np}{n-p}$ if $\frac{n}{n-1}<p<n$ and $\frac{n}{n-1}<r<\infty$ if $p\geq n$. It suffices to prove $\varphi\in L^r$ for our choice of $r$, since in the case $\frac{n}{n-1}<p<n$ we have $\varphi\in L^{\frac{np}{n-p}}$ where $\frac{np}{n-p}>n$, hence we can repeat the argument to complete the proof.

We need the following elliptic apriori estimate for the Dirac operator (see for instance \cite{A03} or \cite{CJSZ2018}): let $\psi$ be a solution to $\D\psi=\chi$ and $\chi\in L^q$ with $1<q<\infty$, then
\eq{\label{apriori_estimate}
C^{-1}\norm{\psi}_{\frac{nq}{n-q}} \leq \norm{\psi}_{W^{1,q}}\leq C \norm{\chi}_q
}
for some $C=C(n,q)>0$.

Given any $M>0$, we denote
\eq{
S_M \coloneqq \sup \left\{ \abs{\int\<\phi,\varphi\>}: \, \norm{\phi}_{r^*}\leq 1, \norm{\phi}_{p^*}\leq M \right\},
}
where $r^*$ is defined by $\frac{1}{r}+\frac{1}{r^*}=1$, so is $p^*$.
It is clear that
\eq{
S_M \leq \norm{\varphi}_r \quad \hbox{and} \quad S_M\leq M\norm{\varphi}_p<\infty.
}
It suffices to prove that
\eq{\label{it_suffices}
S_M \leq C\norm{\varphi}_p, \quad\forall M>0,
}
for a constant $C>0$ independent of $M$.
In fact, since $L^{r^*}\cap L^{p^*}$ is dense in $L^{r^*}$, it follows $\norm{\varphi}_r \leq C(n,p,r)\norm{\varphi}_p$, which finishes the proof.

In order to prove \eqref{it_suffices}, we decompose
\eq{
\lambda A = A_0 + A_\epsilon
}
with
\eq{
\norm{A_0}_\infty \eqcolon C_0<\infty \quad\hbox{and}\quad \norm{A_\epsilon}_n\leq \epsilon
}
for small $\epsilon>0$.
It is possible since $\lambda A\in L^n$. By our assumption ${\rm R}_g>0$, $(M,g)$ does not admit harmonic spinors, hence $\D$ is invertible. Therefore, there exist $\varphi_0$ and $\varphi_\epsilon$ such that
\eq{
\D\varphi_0 = iA_0\cdot\varphi \quad\hbox{and}\quad \D\varphi_\epsilon = iA_\epsilon\cdot\varphi.
}
It follows
\eq{
\D\varphi = i\lambda A\cdot\varphi = i(A_0+A_\epsilon)\cdot\varphi = \D(\varphi_0+\varphi_\epsilon).
}
Hence
\eq{
\varphi = \varphi_0 + \varphi_\epsilon.
}
Given any test spinor $\phi$ satisfying
\eq{\label{test_spinor}
\norm{\phi}_{r^*}\leq 1 \quad\hbox{and}\quad \norm{\phi}_{p^*}\leq M.
}
On one hand, we have
\eq{\label{estimate_varphi_0}
\abs{\int\<\phi,\varphi_0\>} \leq \norm{\varphi_0}_r \leq C_1 \norm{A_0\cdot\varphi}_{\frac{nr}{n+r}} \leq C_2 \norm{A_0}_\infty\cdot\norm{\varphi}_p \leq C_3\norm{\varphi}_p,
}
where we have used \eqref{apriori_estimate}. Here we need $\frac{nr}{n+r}<p$, which is exactly our assumption for $r$.

On the other hand, we have to estimate $\int\<\phi,\varphi_\epsilon\>$. 
Denote
\eq{
\phi_\epsilon \coloneqq iA_\epsilon\cdot\D^{-1}\phi.
}
Note that
\eq{
\norm{\phi_\epsilon}_{r^*} \leq \norm{A_\epsilon}_n\cdot\norm{\D^{-1}\phi}_{\frac{nr^*}{n-r^*}} \leq C_4\epsilon\norm{\phi}_{r^*} \leq C_4\epsilon,
}
where we have used \eqref{apriori_estimate}. Moreover,
\eq{
\norm{\phi_\epsilon}_{p^*}\leq \norm{A_\epsilon}_n\cdot\norm{\D^{-1}\phi}_{\frac{np^*}{n-p^*}}\leq C_5\epsilon\norm{\phi}_{p^*}\leq C_5\epsilon M.
}
Let $C_6 \coloneqq \max\{C_4,C_5\}$. It implies that $(C_6\epsilon)^{-1}\phi_\epsilon$ is a well-defined test spinor satisfying \eqref{test_spinor}. Therefore
\eq{
\abs{\int\<(C_6\epsilon)^{-1}\phi_\epsilon,\varphi\>} \leq S_M
}
and hence
\eq{
\abs{\int\<\phi_\epsilon,\varphi\>} \leq C_6\epsilon S_M.
}
By definition we have
\eq{
\int\<\phi_\epsilon,\varphi\> = \int\<iA_\epsilon\cdot\D^{-1}\phi,\varphi\> = \int\<\D^{-1}\phi,iA_\epsilon\cdot\varphi\> = \int\<\phi,\D^{-1}(iA_\epsilon\cdot\varphi)\> = \int\<\phi,\varphi_\epsilon\>,
}
where we have used the fact that $\D^{-1}$ is self-adjoint. Hence
\eq{\label{estimate_varphi_epsilon}
\abs{\int\<\phi,\varphi_\epsilon\>} \leq C_6\epsilon S_M.
}
Combining \eqref{estimate_varphi_0} and \eqref{estimate_varphi_epsilon} and taking supreme over $\phi$, we have
\eq{
S_M \leq C_3\norm{\varphi}_p + C_6\epsilon S_M.
}
Each constant here depends only on $n,p,r$.
Choosing $\epsilon>0$ small enough completes the proof.
\end{proof}

\

\

\noindent{\sc Acknowledgment.}
M.Z. was supported by CSC of China (No. 202306270117).

\

\noindent{\bf Statements and Declarations} 

\

{\it Data availability}. 
No data was used for the research described in the article.

\

{\it Conflict of interest Statement.} The authors declare that there is no conflict of interest regarding the publication of the work.

\bibliographystyle{alpha}
\bibliography{BibTemplate.bib}

@article{FL22,
  title={Existence of optimizers in a {S}obolev inequality for vector fields},
  author={Frank, R. L. and Loss, M.},
  journal={Ars Inven. Anal.},
  year={2022},
  NUMBER = {1},
lisher={Ars Inveniendi}
}

@article{A03,
  title={A variational problem in conformal spin geometry},
  author={Ammann, B.},
  journal={Habilitationsschift, Universit{\"a}t Hamburg},
  year={2003}
}

@article{WZ25,
  title={Stability of spinorial {S}obolev inequalities on $\mathbb{S}^n$},
  author={Wang, G. and Zhang, M.},
  journal={arXiv preprint arXiv:2508.09047},
  year={2025}
}

@article{WZ25c,
  title={Sobolev inequalities for differential forms on $\mathbb{S}^n$}, 
  author={Wang, G. and Zhang, M.},
  note={preprint},
  year={2025}
}

@article{FL1,
  title={Which magnetic fields support a zero mode?},
  author={Frank, R. L. and Loss, M.},
  journal={Jour. reine angew. Math.},
  volume={2022},
  number={788},
  pages={1--36},
  year={2022},
  publisher={De Gruyter}
}

@book {Baum_Book_90,
    AUTHOR = {Baum, H. and Friedrich, T. and Grunewald, R. and
              Kath, I.},
     TITLE = {Twistor and {K}illing spinors on {R}iemannian manifolds},
    SERIES = {Seminarberichte [Seminar Reports]},
    VOLUME = {108},
 PUBLISHER = {Humboldt Universit\"at, Sektion Mathematik, Berlin},
      YEAR = {1990},
     PAGES = {179},
   MRCLASS = {53C25 (53A50 58G30)}
}

@book{Lawson,
  title={Spin Geometry},
  author={Lawson, H. B. and Michelsohn, M. L.},
  isbn={9780691085425},
  lccn={89032544},
  series={Princeton Mathematical Series},
  url={https://books.google.de/books?id=EWX1vgEACAAJ},
  year={1989},
  publisher={Princeton University Press}
}

@article{Loss_Yau_86,
  title={Stability of {C}oulomb systems with magnetic fields: {III}. {Z}ero energy bound states of the {P}auli operator},
  author={Loss, M. and Yau, H.},
  journal={Communications in mathematical physics},
  volume={104},
  number={2},
  pages={283--290},
  year={1986},
  publisher={Springer}
}

@article{Dunne_Min_08,
  title={Abelian zero modes in odd dimensions},
  author={Dunne, G. V. and Min, H.},
  journal={Physical Review D—Particles, Fields, Gravitation, and Cosmology},
  volume={78},
  number={6},
  pages={067701},
  year={2008},
  publisher={APS}
}

@article{Reuss25, 
title={A note on the existence of nontrivial zero modes on {R}iemannian manifolds},
author={ Reu\ss, J.},
journal={ arXiv preprint arXiv:2503.01602}, year={2025}
}

@article{Jurgen_Julio-Batalla,
  title={Spinorial {Y}amabe-type equations and the {B}{\"a}r-{H}ijazi-{L}ott invariant},
  author={Julio-Batalla, J.},
  journal={J. Geom. Anal.},
  volume={35},
  number={8},
  pages={227},
  year={2025},
  publisher={Springer}
}

@book{Ginoux09,
  title={The {D}irac spectrum},
  author={Ginoux, N.},
  volume={1976},
  year={2009},
  publisher={Springer Science \& Business Media}
}

@article{H86,
  title={A conformal lower bound for the smallest eigenvalue of the {D}irac operator and {K}illing spinors},
  author={Hijazi, O.},
  journal={Comm. Math. Phy.},
  volume={104},
  pages={151--162},
  year={1986},
  publisher={Springer}
}

@article{Friedrich-Kath,
  title={7-dimensional compact {R}iemannian manifolds with {K}illing spinors},
  author={Friedrich, T. and Kath, I.},
  journal={Communications in mathematical physics},
  volume={133},
  number={3},
  pages={543--561},
  year={1990},
  publisher={Springer}
}

@article{CJSZ2018,
  title={Estimates for solutions of {D}irac equations and an application to a geometric elliptic-parabolic problem},
  author={Chen, Q. and Jost, J. and Sun, L. and Zhu, M.},
  journal={Journal of the European Mathematical Society},
  volume={21},
  number={3},
  pages={665--707},
  year={2018}
}

@article{Friedrich_Kim_2000,
  title={The {E}instein-{D}irac equation on {R}iemannian spin manifolds},
  author={Kim, E. and Friedrich, T.},
  journal={Journal of Geometry and Physics},
  volume={33},
  number={1-2},
  pages={128--172},
  year={2000},
  publisher={Elsevier}
}

@book{Friedrich_Book,
  title={Dirac operators in {R}iemannian geometry},
  author={Friedrich, T.},
  volume={25},
  year={2000},
  publisher={American Mathematical Soc.}
}

@article{Bar1993,
  title={Real Killing spinors and holonomy},
  author={B{\"a}r, C.},
  journal={Communications in mathematical physics},
  volume={154},
  number={3},
  pages={509--521},
  year={1993},
  publisher={Springer}
}

@article{McWang1989,
  title={Parallel spinors and parallel forms},
  author={Wang, M.},
  journal={Annals of Global Analysis and Geometry},
  volume={7},
  number={1},
  pages={59--68},
  year={1989},
  publisher={Kluwer Academic Publishers Dordrecht}
}

@article{Frank_Loss_2024,
  title={A sharp criterion for zero modes of the {D}irac equation},
  author={Frank, R. L. and Loss, M.},
  journal={Journal of the European Mathematical Society},
  year={2024}
}

@article{Friedrich-Kath89_JDG,
  title={Einstein manifolds of dimension five with small first eigenvalue of the {D}irac operator},
  author={Friedrich, T. and Kath, I.},
  journal={Journal of differential geometry},
  volume={29},
  number={2},
  pages={263--279},
  year={1989},
  publisher={Lehigh University}
}

@article{Sulanke80,
  title={Der erste {E}igenwert des {D}irac-{O}perators auf $\mathbb{S}^5\slash {\Gamma}$},
  author={Sulanke, S.},
  journal={Mathematische Nachrichten},
  volume={99},
  number={1},
  pages={259--271},
  year={1980},
  publisher={Wiley Online Library}
}

@article{Froelich-Lieb-Loss,
  title={Stability of {C}oulomb systems with magnetic fields: I. {T}he one-electron atom},
  author={Fr{\"o}hlich, J. and Lieb, E. H. and Loss, M.},
  journal={Communications in mathematical physics},
  volume={104},
  number={2},
  pages={251--270},
  year={1986},
  publisher={Springer}
}

@article{Lichnerowicz87,
  title={Spin manifolds, {K}illing spinors and universality of the {H}jazi inequality},
  author={Lichnerowicz, A.},
  journal={Letters in Mathematical Physics},
  volume={13},
  number={4},
  pages={331--344},
  year={1987},
  publisher={Springer}
}

@article {Borrelli21,
    AUTHOR = {Borrelli, W. and Malchiodi, A. and Wu, R.},
     TITLE = {Ground state {D}irac bubbles and {K}illing spinors},
   JOURNAL = {Comm. Math. Phys.},
  FJOURNAL = {Communications in Mathematical Physics},
    VOLUME = {383},
      YEAR = {2021},
    NUMBER = {2},
     PAGES = {1151--1180},
      ISSN = {0010-3616,1432-0916},
   MRCLASS = {53C27 (58E20)},
MRREVIEWER = {Thomas\ Gibbs\ Leness},
       DOI = {10.1007/s00220-021-04013-1},
       URL = {https://doi.org/10.1007/s00220-021-04013-1},
}

@article {Ross-Schroers18,
    AUTHOR = {Ross, C. and Schroers, B. J.},
     TITLE = {Magnetic zero-modes, vortices and {C}artan geometry},
   JOURNAL = {Lett. Math. Phys.},
  FJOURNAL = {Letters in Mathematical Physics},
    VOLUME = {108},
      YEAR = {2018},
    NUMBER = {4},
     PAGES = {949--983},
      ISSN = {0377-9017,1573-0530},
   MRCLASS = {53Z05 (58Z05)},
MRREVIEWER = {Eduardo\ A.\ Gonzalez},
       DOI = {10.1007/s11005-017-1023-2},
       URL = {https://doi.org/10.1007/s11005-017-1023-2},
}

@article{Benguria15,
  title={A criterion for the existence of zero modes for the {P}auli operator with fastly decaying fields},
  author={Benguria, R. D. and Van Den Bosch, H.},
  journal={Journal of Mathematical Physics},
  volume={56},
  number={5},
  year={2015},
  publisher={AIP Publishing}
}

@article {Saito08,
    AUTHOR = {Sait{\"o}, Y. and Umeda, T.},
     TITLE = {The zero modes and zero resonances of massless {D}irac operators},
   JOURNAL = {Hokkaido Math. J.},
  FJOURNAL = {Hokkaido Mathematical Journal},
    VOLUME = {37},
      YEAR = {2008},
    NUMBER = {2},
     PAGES = {363--388},
      ISSN = {0385-4035},
   MRCLASS = {35Q40 (35B35 35P15 47F05 47N50 81Q10)},
MRREVIEWER = {Karl\ Michael\ Schmidt},
       DOI = {10.14492/hokmj/1253539560},
       URL = {https://doi.org/10.14492/hokmj/1253539560},
}

@article {Elton02,
    AUTHOR = {Elton, D. M.},
     TITLE = {The local structure of zero mode producing magnetic potentials},
   JOURNAL = {Comm. Math. Phys.},
  FJOURNAL = {Communications in Mathematical Physics},
    VOLUME = {229},
      YEAR = {2002},
    NUMBER = {1},
     PAGES = {121--139},
      ISSN = {0010-3616,1432-0916},
   MRCLASS = {81Q10 (35Q60 47F05 47N20 81V10)},
MRREVIEWER = {Michael\ Melgaard},
       DOI = {10.1007/s00220-002-0679-2},
       URL = {https://doi.org/10.1007/s00220-002-0679-2},
}

@article {Erdoes-Solovej01,
    AUTHOR = {Erd{\H{o}}s, L. and Solovej, J. P.},
     TITLE = {The kernel of {D}irac operators on $\mathbb{S}^3$ and $\mathbb{R}^3$},
   JOURNAL = {Rev. Math. Phys.},
  FJOURNAL = {Reviews in Mathematical Physics. A Journal for Both Review and
              Original Research Papers in the Field of Mathematical Physics},
    VOLUME = {13},
      YEAR = {2001},
    NUMBER = {10},
     PAGES = {1247--1280},
      ISSN = {0129-055X,1793-6659},
   MRCLASS = {58J60 (47F05 53C27 81Q10)},
MRREVIEWER = {Hans-Bert\ Rademacher},
       DOI = {10.1142/S0129055X01000983},
       URL = {https://doi.org/10.1142/S0129055X01000983},
}

@article {Balinsky-Evans01,
    AUTHOR = {Balinsky, A. A. and Evans, W. D.},
     TITLE = {On the zero modes of {P}auli operators},
   JOURNAL = {J. Funct. Anal.},
  FJOURNAL = {Journal of Functional Analysis},
    VOLUME = {179},
      YEAR = {2001},
    NUMBER = {1},
     PAGES = {120--135},
      ISSN = {0022-1236,1096-0783},
   MRCLASS = {47N50 (35Q40 47F05 81Q10)},
MRREVIEWER = {Mikl\'os\ Horv\'ath},
       DOI = {10.1006/jfan.2000.3670},
       URL = {https://doi.org/10.1006/jfan.2000.3670},
}

@article {Aharonov-Casher19,
    AUTHOR = {Aharonov, Y. and Casher, A.},
     TITLE = {Ground state of a spin-$\frac{1}{2}$ charged particle in a two-dimensional magnetic field},
   JOURNAL = {Phys. Rev. A (3)},
  FJOURNAL = {Physical Review. A. Third Series},
    VOLUME = {19},
      YEAR = {1979},
    NUMBER = {6},
     PAGES = {2461--2462},
      ISSN = {1050-2947,1094-1622},
   MRCLASS = {81C05 (58G10)},
       DOI = {10.1103/PhysRevA.19.2461},
       URL = {https://doi.org/10.1103/PhysRevA.19.2461},
}

@article{Obata62,
  title={Certain conditions for a {R}iemannian manifold to be isometric with a sphere},
  author={Obata, M.},
  journal={Journal of the Mathematical Society of Japan},
  volume={14},
  number={3},
  pages={333--340},
  year={1962},
  publisher={The Mathematical Society of Japan}
}

@article{Obata71,
  title={The conjectures on conformal transformations of {R}iemannian manifolds},
  author={Obata, M.},
  journal={Bull. Amer. Math. Soc.},
  volume={77},
  number={6},
  pages={265--270},
  year={1971}
}

@article {Herzlich-M,
    AUTHOR = {Herzlich, M. and Moroianu, A.},
     TITLE = {Generalized {K}illing spinors and conformal eigenvalue
              estimates for {${\rm Spin}^c$} manifolds},
   JOURNAL = {Ann. Global Anal. Geom.},
  FJOURNAL = {Annals of Global Analysis and Geometry},
    VOLUME = {17},
      YEAR = {1999},
    NUMBER = {4},
     PAGES = {341--370},
      ISSN = {0232-704X,1572-9060},
   MRCLASS = {58J50 (53C27 58J60)},
MRREVIEWER = {Jean-Louis\ Milhorat},
       DOI = {10.1023/A:1006546915261},
       URL = {https://doi.org/10.1023/A:1006546915261},
}

@book{Boyer_G_Book,
  title={Sasakian geometry},
  author={Boyer, C. and Galicki, K.},
  year={2007},
  publisher={Oxford university press}
}

@article{Lee_Parker,
  title={The {Y}amabe problem},
  author={Lee, J. M. and Parker, T. H.},
  journal={Bulletin of the American Mathematical Society},
  volume={17},
  number={1},
  pages={37--91},
  year={1987}
}

@article{Grosse23, title={A Note on the spectrum of magnetic {D}irac operators}, author={Charalambous, N. and Gro{\ss}e, N.}, journal={SIGMA. Symmetry, Integrability and Geometry: Methods and Applications}, volume={19}, pages={102}, year={2023}, publisher={SIGMA. Symmetry, Integrability and Geometry: Methods and Applications} }

@article{Riviere24,
  title={Integrability by compensation for {D}irac equation},
  author={Da Lio, F. and Rivi{\`e}re, T. and Wettstein, J.},
  journal={Transactions of the American Mathematical Society},
  volume={375},
  number={6},
  pages={4477--4511},
  year={2022}
}

@article {Riviere20,
    AUTHOR = {Da Lio, Francesca and Rivi\`ere, Tristan},
     TITLE = {Critical chirality in elliptic systems},
   JOURNAL = {Ann. Inst. H. Poincar\'e{} C Anal. Non Lin\'eaire},
  FJOURNAL = {Annales de l'Institut Henri Poincar\'e{} C. Analyse Non
              Lin\'eaire},
    VOLUME = {38},
      YEAR = {2021},
    NUMBER = {5},
     PAGES = {1373--1405},
      ISSN = {0294-1449,1873-1430},
   MRCLASS = {35J47 (20G20 35B65)},
  MRNUMBER = {4300926},
       DOI = {10.1016/j.anihpc.2020.11.006},
       URL = {https://doi.org/10.1016/j.anihpc.2020.11.006},
}

@article{BGH25,
  title={On the spectrum of the magnetic {D}irac operator},
  author={Branding, V. and Ginoux, N. and Habib, G.},
  journal={arXiv preprint arXiv:2512.13338},
  year={2025}
}

@article{Bar1992,
  title={Lower eigenvalue estimates for {D}irac operators},
  author={B{\"a}r, C.},
  journal={Mathematische Annalen},
  volume={293},
  number={1},
  pages={39--46},
  year={1992},
  publisher={Springer}
}

\end{document}